\documentclass[11pt]{article}

\usepackage{amsmath,amsthm,amssymb,mathrsfs} 
\usepackage{times}
\usepackage[colorlinks,citecolor=red,urlcolor=blue,pagebackref,hypertexnames=false]{hyperref}
\usepackage{fancybox,fancyhdr,graphics,epsfig}
\usepackage[usenames,dvipsnames]{color}
\usepackage{datetime2}
\usepackage[utf8]{inputenc}

\def\bL{\mathbf{L}}

\def\cF{\mathcal{F}}
\def\cL{\mathcal{L}}
\def\bM{\mathbf{M}}
\def\bN{\mathbf{N}}

\def\bK{\mathbf{K}}
\def\dkr{\mathbf{d_{KR}}}
\def\dTV{\mathbf{d_{TV}}}
\def\txi{\tilde{\xi}}
\def\no{\nonumber}

\def\cLip{\textsc{Lip}}
\usepackage{amsmath,amsthm,bbm,amssymb}
\usepackage[dvipsnames]{xcolor}

\DeclareMathOperator{\dist}{dist}

\DeclareMathOperator{\vol}{Vol}
\DeclareMathOperator{\logg}{\log\log}

\def\N{\mathbb{N}}

\def\R{\mathbb{R}}

\def\Z{\mathbb{Z}}
\def\P{\mathbb{P}}

\def\T{\mathbb{T}}
\def\X{\mathbb{X}}
\def\Y{\mathbb{Y}}

\def\cF{\mathcal{F}}

\def\cL{\mathcal{L}}

\def\cN{\mathcal{N}}

\def\cS{\mathcal{S}}

\def\cX{\mathcal{X}}
\def\cY{\mathcal{Y}}

\def\cL{\mathcal{L}}

\newcommand{\E}{\mathbb{E}} 

\newcommand{\given}{\;|\;}
\newcommand{\mean}[1] {\E\left\{{#1}\right\}}

\newcommand{\meanx}[1] {\E\{{#1}\}}
\newcommand{\cmean}[2] {\E\left\{#1\given #2\right\}}

\newcommand{\ind}{\boldsymbol{\mathbbm{1}}} 

\newcommand{\var}[1]{\mathrm{Var}\param{{#1}}}
\newcommand{\varx}[1]{\mathrm{Var}({#1})}

\newcommand{\set}[1]{\left\{#1\right\}}
\newcommand{\sbrk}[1]{\left[#1\right]}

\newcommand{\param}[1]{\left(#1\right)}
\newcommand{\abs}[1] {\left| {#1}\right|}

\newcommand{\prob}[1]{\mathbb{P}\left(#1\right)}

\newcommand{\tg}{\tilde{g}}
\newcommand{\eps}{\epsilon}
\newcommand{\bz}{\mathbf{z}}
\newcommand{\by}{\mathbf{y}}
\newcommand{\bx}{\mathbf{x}}

\providecommand{\setthms}[1]{#1}
\setthms{
\newtheorem{lem}{Lemma}[section]
\newtheorem{thm}[lem]{Theorem}
\newtheorem{prop}[lem]{Proposition}

\newtheorem{rem}[lem]{Remark}

\theoremstyle{definition}

}

\newcommand{\ninf}{n\to\infty}

\newcommand{\limninf}{\lim_{\ninf}}

\numberwithin{equation}{section}

\def\bsplit#1\esplit{\begin{split} #1 \end{split} }
\def\splitb#1\splite{\begin{split} #1 \end{split} }
\def\beq#1\eeq{\begin{equation} #1 \end{equation}}
\def\eqb#1\eqe{\begin{equation} #1 \end{equation}}

\def\md{\mathrm{d}}

\newcommand{\remove}[1]{}

    \newcommand{\footremember}[2]{%
    \footnote{#2}
    \newcounter{#1}
    \setcounter{#1}{\value{footnote}}%
    }
    
\title{Poisson process approximation under stabilization and Palm coupling} 
\author{%
O. Bobrowski \footremember{ob}{Viterbi Faculty of Electrical and Computer Engineering, Technion - Israel Institute of Technology, omer@ee.technion.ac.il.} %
\and M. Schulte \footremember{ms}{Institute of Mathematics, Hamburg University of Technology, matthias.schulte@tuhh.de.} %
\and D. Yogeshwaran \footremember{dy}{Theoretical Statistics and Mathematics Unit,  Indian Statistical Institute,  Bangalore.  d.yogesh@isibang.ac.in}  
}
\date{\today}

\begin{document}

\maketitle
\renewcommand{\baselinestretch}{1.1}\normalsize

\begin{abstract}
We present new Poisson process approximation results for stabilizing functionals of Poisson and binomial point processes. These functionals are allowed to have an unbounded range of interaction and encompass many examples in stochastic geometry.  Our bounds are derived for the Kantorovich-Rubinstein distance using the generator approach to Stein's method. We give different types of bounds for different point processes. While some of our bounds are given in terms of coupling of the point process with its Palm version, the others are in terms of the local dependence structure formalized via the notion of  stabilization. We provide two supporting examples for our new framework -- one is for  Morse critical points of the distance function, and the other is for large $k$-nearest neighbor balls.  Our bounds considerably extend the results in Barbour and Brown (1992),  Decreusefond,  Schulte and Th\"{a}le (2016) and Otto (2020).
\end{abstract}
\vspace{0.1cm}

\section{Introduction} 

Poisson processes arise in many fields of probability theory and are arguably the most prominent class of point processes. At the same time they are very convenient to work with since they exhibit complete spatial independence (i.e., point configurations within disjoint sets are independent). Therefore, approximating an involved point process by a Poisson process is an important problem. This is also the crucial step in establishing convergence in distribution to a Poisson process. This paper focuses on Poisson process approximation for a large class of point processes that often arise in stochastic geometry. The processes we study are functionals of Poisson (or binomial) point processes, which are themselves not Poisson, and in particular lack spatial independence.  The results we present here significantly generalize recent ones \cite{decreusefond_functional_2016,Otto2020} by either considering a stronger approximation distance or more general functionals.  Our approach is based on Stein's method for Poisson process approximation; see e.g.\ \cite{Barbour1988,Barbour1992,Barbour1992book,Chen2004,decreusefond_functional_2016,
Xia2005}. Now we shall offer a quick preview of our setup,  results and some background literature.

Let $\nu$ be either a Poisson process or a binomial point process (i.e., a collection of $n$ i.i.d.\ points) in some space $\mathbb{X}$. We denote by $\mathbf{N}_{\mathbb{X}}$ the set of $\sigma$-finite point configurations on $\mathbb{X}$. For $k\in\N$, a second space $\mathbb{Y}$ and measurable functions $f: \mathbb{X}^k\times \mathbf{N}_{\mathbb{X}}\to \mathbb{Y}$ and $g: \mathbb{X}^k\times \mathbf{N}_{\mathbb{X}}\to \{0,1\}$ that are symmetric in the first $k$ arguments, we consider the following point process on $\mathbb{Y}$,
\begin{equation}\label{e:Definition_Point_Process}
\xi {:=} \xi[\nu] {:=} \frac{1}{k!}\sum_{\bx \in \nu^k_{\neq}} g(\bx,\nu) \delta_{f(\bx,\nu)},
\end{equation}
where $\nu^k_{\neq}$ denotes the set of all $k$-tuples of distinct points of $\nu$, and $\delta_y$ stands for the Dirac measure concentrated at $y\in\mathbb{Y}$.  
Point processes of the form \eqref{e:Definition_Point_Process} often arise in stochastic geometry (cf. \cite{Otto2020,PianoforteSchulte2021}).  For example, consider the case $\mathbb{X}=\mathbb{R}^d$, $\mathbb{Y}=\mathbb{R}$, $k=1$, $g \equiv 1$ and $f(x,\nu)$ being the minimal distance from $x$ to another point of $\nu$. The point process $\xi$ is then the collection of all nearest neighbor distances of $\nu$.  

Our main results deal with the comparison of $\xi$ with a Poisson process $\zeta$ on $\mathbb{Y}$ with a finite intensity measure. We consider the Kantorovich-Rubinstein distance between the distributions of the two processes $\xi$ and $\zeta$, defined as
\begin{equation}
\label{d:dkr}
\dkr(\xi,\zeta) {:=} \sup_{h \in \operatorname{LIP}(\mathbb{Y})} |\mean{h(\xi)} - \mean{h(\zeta)}|,
\end{equation}
where $\operatorname{LIP}(\mathbb{Y})$ is the class of {measurable} $1$-Lipschitz functions with respect to the total variation distance on the space of finite point configurations on $\mathbb{Y}$ ({see \eqref{e.hlip_gen}}). Note that convergence in Kantorovich-Rubinstein distance implies convergence in distribution, and that the Kantorovich-Rubinstein distance dominates the total variation distance
\begin{equation}
\label{d:dtv}
\dTV(\xi,\zeta) {:=} \sup_{A\in \mathcal{N}_{\mathbb{Y}}} |\prob{\xi\in A} - \prob{\zeta\in A}|,
\end{equation}
where $\mathcal{N}_{\mathbb{Y}}$ is the $\sigma$-field on the space $\mathbf{N}_{\mathbb{Y}}$. Studying the Kantorovich-Rubinstein distance goes in the opposite direction to many other works on Poisson process approximation via Stein's method, where weaker distances than the total variation distance were considered, see e.g.\ \cite{Barbour1992, Barbour1992book, Chen2004, Schuhmacher2009, Xia2005}.

In general,  if the behaviour of $\xi$ is close to that of a Poisson process, by taking  $\mathbb{Y}=\mathbb{R}$ one gets an approximation for
$$
\max \{f(\bx,\nu): \bx\in\nu^k_{\neq}, g(\bx,\nu)=1\}.
$$ 
Thus,  approximation results for the point process $\xi$ can be a crucial tool to compute the limiting distributions of some extreme values of $\xi$.  In a similar fashion,  we can treat functionals of $\xi$ other than the maximum,  such as order statistics,  sums of the points or number of points in a certain region.

This paper continues the line of research initiated in \cite{decreusefond_functional_2016}. There,  processes of the form \eqref{e:Definition_Point_Process} were studied, but for a special case where $g$ and $f$ 	 depend only on $\bx$ and not on $\nu$, i.e.,  $f(\bx,\nu) = f(\bx),  g(\bx,\nu) = g(\bx)$.  In other words, the results of \cite{decreusefond_functional_2016} allow to deal with point processes that have the form of $U$-statistics. However, this formulation excludes point processes that arise naturally such as the nearest neighbor example mentioned above. This lacuna is addressed in the present paper.  We summarize our main results as follows.
\begin{itemize}
\item We establish a general Poisson process approximation result (Theorem \ref{t:gen_Poisson_approx}), which extends \cite[Theorem 2.6]{Barbour1992} from the total variation distance to the Kantorovich-Rubinstein distance.  Here we do not assume any structure on $\xi$ and the bounds are given in terms of a coupling between $\xi$ and its Palm measure.  

\item A similar result is derived for $\xi[\nu]$ under the assumption that $\nu$ is a Poisson process and that $f$ in \eqref{e:Definition_Point_Process} depends only on $\bx$ (see Theorem \ref{t:KR_coupling}). This result already extends the Poisson process approximation result of \cite[Theorem 3.1]{decreusefond_functional_2016} and may be considered as the process-level analogue of the Poisson approximation result in \cite[Theorem 3.1] {penrose_inhomogeneous_2018}. 
Theorem \ref{t:gen_Poisson_approx} and Theorem \ref{t:KR_coupling} already cover many existing applications in the literature (see \cite{decreusefond_functional_2016,Otto2020}). 
   
\item We establish a Poisson process approximation result for $\xi[\nu]$ when $\nu$ is a Poisson process and $f(\bx,\nu), g(\bx,\nu)$ depend on random neighborhoods around $\bx$ (see Theorem \ref{t:qldep}).  This is an important concept in stochastic geometry known as stabilization.  We  elaborate on this {in Remark \ref{r:stab}}.  Here,  our bounds are given more explicitly in terms of $g$ as well as the tail probabilities of the random neighborhoods.  This result in its full generality is one of the novel contributions of this work and an analogous result even in total variation distance is not available to the best of our knowledge.  Under total variation distance,  approximation results are proven only for $f(\bx,\nu) = f(\bx)$ in \cite[Theorem 3.3]{Otto2020} and the bounds therein are same as ours (barring a factor of $2$) even though our bounds hold for more general functions $f$ and in the stronger Kantorovich-Rubinstein distance. 

\item For the binomial case, the lack of spatial independence leads to a more restricted statement and more involved bounds. In particular, for $k = 1$ we can show a Poisson process approximation result under {stabilization} of $f(x,\nu), g(x,\nu)$ (see Theorem \ref{thm:binomial}). The restriction to $k=1$ here is more for ease of stating the bounds, see Remark \ref{r:krbin}.  We are not aware of such a general approximation result even under weaker distances for point processes driven by stabilizing functionals of binomial point {processes.}  In fact,  such results are not known even for functionals with finite range of stabilization, i.e., {$f(x,\nu), g(x,\nu)$} depend on fixed compact neighborhoods around $\bx$.
\end{itemize}

One point of difference between the four results is that the bounds in the first two results do not assume any local dependence structure on $f$ and $g$, and the bounds are given in terms of a coupling.  On the other hand, the latter two bounds are given explicitly in terms of of the functional $g$ alone.  As will be illustrated {in Section \ref{sec:Applications},} our general bounds shall necessitate a few additional computations in applications compared to those for a Poisson {(non-functional)} approximation result.  

\begin{rem} 
\label{r:stab}
 The notion of stabilization was introduced in the context of proving law of large numbers and central limit theorems for functionals of point processes \cite{Schreiber2010,Yukich2013},  and is meant to capture the spatial dependence of the functionals.  Apart from its extensive use in proving laws of large numbers and central limit theorems,  it has also been used to prove moderate deviations \cite{Eichelsbacher2015} and normal approximation bounds \cite{Lachieze2019}.  Even though this notion is implictly used in some of the Poisson approximation results,  to the best of our knowledge there has not been a general Poisson process approximation result for stabilizing functionals until the recent work  \cite{Otto2020}.  In this regards,  our article makes an important contribution in furthering the usage of stabilization in Poisson process approximation.  More detailed comparisons of our results with those in the literature (especially \cite{Barbour1992,decreusefond_functional_2016,Otto2020}) are provided after the respective theorem statements. 
\end{rem}

 Apart from extending total variation bounds in \cite{Otto2020} to Kantorovich-Rubinstein distance and the applications therein,  we further envisage that the general Poisson process approximation result (Theorem \ref{t:gen_Poisson_approx}) and the proof approach may be useful for deriving similar bounds when $\nu$ is not a Poisson process as well.   Our latter results (Theorems \ref{t:qldep} and \ref{thm:binomial}),  significantly enhance the scope of many applications and also allow us to consider more general point processes than those investigated in the literature so far.  

In order to demonstrate the applicability and versatility of our results,  we consider two examples -- one from the field of random topology and the other from stochastic geometry:   
(1)   Morse critical points: For a homogeneous Poisson process on a flat torus, we consider the distance function (from the points) and define critical points together with their Morse index.  Grouping together all critical points of the same index, we have a point process for which we wish to prove Poisson convergence under suitable scaling. This convergence statement has a significant contribution to the analysis of the homological connectivity phenomenon studied in \cite{Bobrowski2019}. In particular, it yields the asymptotic behaviour of the persistent homology  in the critical window for homological connectivity (see Remark \ref{r:critpts}). (2)  Large $k$-nearest neighbor balls:  For an underlying Poisson or binomial point process we put around each point a ball whose radius is the distance to its $k$-nearest neighbor and establish that {the scaled volumes (with respect to the intensity measure)} of these balls can be approximated by a Poisson process.  In both the applications, our results yield explicit bounds for the rate of convergence (when the point processes are restricted to suitable sets).  Here,  by a more careful computation,  we obtain better bounds than those in \cite[Theorem 4.2]{Otto2020} for {underlying Poisson processes} and we also obtain a new approximation result when the underlying point process $\nu$ is a {binomial} point process.  Comparisons between our applications and those existing in the literature as well as more potential applications are discussed in Section \ref{sec:Applications}.

Before we end the introduction,  we will say a quick word about our proofs.  In order to control the Kantorovich-Rubinstein distance between $\xi$ and $\zeta$, we employ the same generator approach as in \cite{decreusefond_functional_2016}, see also \cite{Barbour1992}.  To further bound the resulting expressions we use two classical approaches from Stein's method,  coupling and local dependence,  which lead to slightly different results for the Poisson process approximation of $\xi$.  

This paper is organized as follows. After introducing some preliminaries in Section \ref{sec:Preliminaries}  (including Stein's method for Poisson process approximation), we present our coupling approach in Section \ref{sec:Coupling}.  Local dependence (and stabilization) is used for underlying Poisson and binomial point processes in Sections \ref{s:locdep} and \ref{s:locdep_binomial}, respectively.  Finally,  we provide two applications of our  main results to {random topology and} stochastic geometry in Section \ref{sec:Applications}.  

\section{Preliminaries}\label{sec:Preliminaries}
\subsection{Point processes}

We shall adapt the notation from \cite{decreusefond_functional_2016} and for more on point process theory, we refer the reader to \cite{Baccelli2020,LastPenrose17}.
Let  $(\X, \cX)$ be a locally compact second countable Hausdorff space (lcscH space).
We denote by $\bN_{\X}$ the space of $\sigma$-finite counting measures on $\X$, and by $\widehat \bN_{\X}$ the space of all finite counting measures on $\X$. We equip the spaces $\bN_{\X}$ and $\widehat{\bN}_{\X}$ with the corresponding $\sigma$-algebras $\cN_{\X}$ and $\widehat{\cN}_{\X}$, {which} are induced by the maps $\omega \mapsto \omega(B)$ {for all $B \in \cX$.} 

By \cite[Proposition 6.2]{LastPenrose17}, each $\omega\in\bN_{\X}$ can be written as $\omega=\sum_{i\in I} \delta_{x_i}$, where $(x_i)_{i\in I}$ is a countable collection of points in $\X$ and $\delta_x$ stands for the Dirac measure concentrated at $x\in\X$. Due to this representation we can think of counting measures as multisets and treat them as such by abusing notation.  Thus, we write $\sum_{x \in\omega}$ for the sum over all $x_i$, $i\in I$, which is the same as the integral with respect to $\omega$.  More generally, we define for $k\in\N$,
\remove{$$
\omega^k_{\neq} := \{ (y_1,\hdots,y_k)\in \omega^k : y_j\neq y_\ell, j\neq \ell \} := \{ (x_{i_1},\hdots,x_{i_k}): i_1,\hdots,i_k\in I, i_j\neq i_\ell, j\neq \ell \}
$$
}
$$
\omega^k_{\neq} := \{ (x_{i_1},\hdots,x_{i_k}): i_1,\hdots,i_k\in I, i_j\neq i_\ell, j\neq \ell \}
$$
and denote by $\sum_{(x_1,\hdots,x_k)\in \omega^k_{\neq}}$ the sum over all $k$-tuples of distinct points of $\omega$, which is the same as the integral with respect to the {$k$-th} factorial measure of $\omega$ (see \cite[Subsection 4.2]{LastPenrose17}). By $\omega\cap S$ we mean the restriction of $\omega$ to $S\in\cX$. Moreover,  we write $\omega_1\subset\omega_2$ for $\omega_1,\omega_2\in\bN_{\X}$ if $\omega_1(A)\leq \omega_2(A)$ for all $A\in\cX$. 
In addition, for $\omega_1,\omega_2\in\bN_{\X}$ such that $\omega_1=\sum_{i\in I} \delta_{x_i}$ and $\omega_2=\sum_{j\in J} \delta_{y_j}$, we define
$$
\omega_1\setminus\omega_2 := \sum_{x\in \{x_i: i\in I\}\setminus \{y_j: j\in J\}} \delta_x.
$$
The sets under the sum must be understood as multisets since $\omega_1$, $\omega_2$ and $\omega_1\setminus\omega_2$ can have multiple points (i.e., points $x\in \X$ such that $\omega(\{x\}) > 1$). The symmetric difference $\omega_1\triangle\omega_2$ is defined as $\omega_1\triangle\omega_2:=\omega_1\setminus\omega_2 + \omega_2\setminus\omega_1$.

The total variation distance between two measures $\mu_1,\mu_2$ on $\X$ is defined as 
$$
d_{TV}(\mu_1,\mu_2):= \sup_{\substack{A \in \cX\\ \mu_1(A),\mu_2(A)<\infty}}|\mu_1(A) - \mu_2(A)|.
$$
We denote by $\cLip(\X)$ the class of measurable functionals $h : \widehat{\bN}_{\X} \to \R$ such that for all $\omega_1,\omega_2 \in  \widehat{\bN}_{\X}$,
\begin{equation}
\label{e.hlip_gen}
|h(\omega_1) - h(\omega_2)| \leq d_{TV}(\omega_1,\omega_2) 
\end{equation}
In other words, $\cLip(\X)$ is the class of measurable $1$-Lipschitz functionals on $\widehat{\bN}_{\X}$ with respect to the total variation distance between measures. Let $\|\cdot\|$ denote  the {\em total variation} of a signed measure, i.e., $\|\mu\| = \mu_+(\X) + \mu_-(\X)$ where $\mu_+,\mu_-$ are the positive and negative parts of the signed measure $\mu$. 
Note that if $\omega_1,\omega_2 \in \widehat \bN_{\X}$, then $d_{TV}(\omega_1,\omega_2) = \max \{ (\omega_1 \setminus \omega_2)(\X), (\omega_2 \setminus \omega_1)(\X) \}$.
Hence, we have the following relation between the total variation distance and the total variation of the difference between counting measures
\begin{equation}
\label{e:dtvbd}
 d_{TV}(\omega_1,\omega_2) \leq \|\omega_1 - \omega_2\| = (\omega_1 \triangle \omega_2)(\X) \leq 2d_{TV}(\omega_1,\omega_2).
\end{equation}
The first inequality becomes an equality when $\omega_1 \subset \omega_2$. 
Thus, measuring distance between counting measures via total variation distance and total variation of measures differs at most by a factor of $2$. 

\remove{We will want to measure the distance between finite point processes, i.e., random counting measures in $\widehat\bN_{\X}$. For such processes $\zeta_1,\zeta_2$, we define the Kantorovich-Rubinstein (KR) distance $\dkr$ \cite[(2.6)]{decreusefond_functional_2016} as
\begin{equation}
\label{d:dkr}
\dkr(\zeta_1,\zeta_2) := \sup_{h \in \cLip(\X)}|\mean{h(\zeta_1)} - \mean{h(\zeta_2)}|.
\end{equation}
The more classical distance between point processes is the total variation distance. This is defined as
$$\dTV({\zeta_1},{\zeta_2}) :=  \sup_{A \in \cN_{\X}} |\prob{\zeta_1 \in A} - \prob{\zeta_2 \in A}|.$$
}
Recall the definitions of Kantorovich-Rubinstein (KR) distance $\dkr$ and the classical total variation distance $\dTV$ in \eqref{d:dkr} and \eqref{d:dtv} respectively.  We refer to \cite[Section 2.5]{decreusefond_functional_2016} for a dual formulation as well as more details on the KR distance.  Since, for $A\in\cX$, $\omega \mapsto \ind\{\omega \in A\}$ is a function in $\cLip(\X)$, we have that  $\dTV(\zeta_1,\zeta_2) \leq \dkr(\zeta_1,\zeta_2)$ for finite point processes $\zeta_1$ and $\zeta_2$. 
For an example of a sequence of point processes that converges in $\dTV$ but not in $\dkr$, see \cite[Example 2.2]{decreusefond_functional_2016}.

Let $\zeta$ be a point process with $\sigma$-finite intensity measure $\bM$, i.e., $\mean{\zeta(A)}=\bM(A)$ for all $A\in\cX$. The {\em reduced Palm expectation} of $\zeta$ at $x$ denoted by $\E^{!}_{x}$ is defined via the Campbell-Mecke-Little formula \cite[Corollary 3.1.14]{Baccelli2020},  as follows. For a measurable function $f : \X \times \bN_{\X} \to \R_+$, it holds that
\begin{equation}
\label{e:clm}
\mean{\sum_{x \in \zeta} f(x,\zeta)} = \int_{\X}\E^{!}_{x}\big\{f(x,\zeta +\delta_{x})\big\} \bM(\md x).
\end{equation}
The above formula can be extended to integrable real-valued functions by standard measure-theoretic arguments. Strictly speaking, $\E^{!}_{x}$ is defined only for $\bM$-a.e.\ $x$, but this suffices for our purposes. The corresponding probability distribution is denoted by $\P^{!}_{x}$ and the point process with this probability distribution is denoted by $\zeta^{!}_x$. 

An essential tool for studying Poisson processes is the multivariate Mecke equation (see \cite[Theorem 4.4]{LastPenrose17}). For a Poisson process $\eta$ on $\X$ with $\sigma$-finite intensity measure $\bK$ and a measurable function $f : \X^k \times \bN_{\X} \to \R_+$, we have that 
\begin{equation}
\label{e:mmecke}
\mean{\sum_{\bx \in \eta^k_{\ne}} f(\bx,\eta)} = \int_{\X^k}\mean{f(\bx,\eta+\delta_{\bx})}\bK^k(\md \bx),
\end{equation}
where $\bx = (x_1,\ldots,x_k)$ and $\delta_{\bx} := \sum_{i=1}^k \delta_{x_i}$. The Mecke equation for $k = 1$ characterizes the Poisson process \cite[Theorem 4.1]{LastPenrose17} and by the definition of the reduced Palm version of a point process, we can also rephrase this characterization as $\eta^{!}_x \stackrel{d}{=} \eta$ for all $x \in \X$,  where $\stackrel{d}{=}$ denotes equality in distribution of the two random elements.

For two Poisson processes $\zeta_1$ and $\zeta_2$ with finite intensity measures $\bM$ and $\bL$, it is known (see \cite[Remark 3.2(iv)]{decreusefond_functional_2016}) that
\begin{equation}
\label{e:dkrPoisson0}
\dkr(\zeta_1,\zeta_2)\leq d_{TV}(\bM,\bL).
\end{equation}
Thus, it follows from the triangle inequality for $\dkr$ that, for any finite point process $\xi$,
\begin{equation}\label{e:dkrPoisson}
\dkr(\xi,\zeta_1) \leq d_{TV}(\bM,\bL) + \dkr(\xi,\zeta_2).
\end{equation}
Due to this inequality it is often sufficient to compare a finite point process with a Poisson process with the same intensity measure.

\subsection{Stein's method for Poisson process approximation}\label{sec:prelim_stein}

Stein's method is used to compare the distribution of two random objects (cf. \cite{Ross2011}). In our case we would like to compare a  point process $\xi$ to a given Poisson process $\zeta$. We will use the generator approach with respect to the KR distance as in \cite{decreusefond_functional_2016}, using Glauber dynamics for Poisson process. See e.g.\ \cite{Barbour1992, Barbour1992book, Chen2004, Schuhmacher2009, Xia2005} for more on Poisson process approximation under other metrics.

Let $\zeta$ be a Poisson process on $(\X,\cX)$, with a finite intensity measure $\bM$.
For $h:\widehat\bN_{\X}\to \R$ {and $\omega\in\widehat\bN_{\X}$,} we define
\begin{equation}
\label{d:gldyn}
\cL h(\omega) := \int_{\X} D_xh(\omega) \bM(\md x) - \int_{\X} D_xh(\omega - \delta_x) \omega(\md x),
\end{equation}
where $D_x h(\omega) := h(\omega+\delta_x) - h(\omega)$, and $\int_{\X} f(x)\omega(\md x) = \sum_{x\in \omega} f(x)$. 
It can be shown (see \cite[Proposition 10.4.VII]{Daley2007} or \cite{Preston1975}) that the operator $\cL$ is the generator of a Markov process whose stationary distribution is the same as that of $\zeta$.
Further, $\cL$ satisfies the following equation (see \cite[(4.3)]{decreusefond_functional_2016})
\begin{equation}
\label{e:stein}
 \mean{h(\zeta)} - h(\omega) = \int_0^{\infty}\cL P_sh(\omega) \md s,
\end{equation}
where $P_s$ denotes the Markov semigroup (also known as the Ornstein-Uhlenbeck semigroup) corresponding to the generator $\cL$, and $P_0h \equiv h$. The Markov process $(\zeta_s)_{s \geq 0}$ corresponding to the generator $\cL$ is called Glauber dynamics and is the spatial birth-death process in continuous time described as follows. If $\zeta_s = \omega$ at time $s$, each particle in $\omega$ dies at rate $1$ and new particles are born at rate $\bM(\X)$ with their location chosen according to the probability measure $\bM(\cdot) / \bM(\X)$. The new particles also have an exponentially distributed (with mean $1$) lifetime. With this process representation, we can represent the semigroup $P_s$ as follows. For $h:\widehat\bN_{\X}\to \R$ we have
$$ P_sh(\omega) = \mean{h(\zeta_s) \mid  \zeta_0 = \omega}.$$

An important property we will use later is that $P_s$ satisfies the following  Lipschitz property with respect to the total variation of measures \cite[Lemma 5.2]{decreusefond_functional_2016},
\begin{equation}
\label{e:pslip}
|P_sh(\omega_1) - P_sh(\omega_2)| \leq e^{-s}(\omega_1 \triangle \omega_2)(\X),
\end{equation}
for $\omega_1,\omega_2 \in \widehat{\bN}_{\X}$ and $h \in \cLip(\X)$. The property is deduced using the above representation of the semigroup and the exponential lifetimes of the particles. 

\section{Approximation for point processes via Palm coupling}\label{sec:Coupling}
We present two results in this section that provide Poisson process approximation bounds in the KR distance via suitable coupling between the original process and its Palm measure.  While our first theorem holds for general point processes, the second theorem is specialized to the case of point processes driven by a Poisson process. These two results already significantly generalize many of the existing bounds in the literature. We shall discuss these connections in detail after the respective theorem statements.
\begin{thm}
\label{t:gen_Poisson_approx}
Let $\xi$ be a finite point process with intensity measure $\bL$ on $\X$, such that $\bL(\X) < \infty$. Let $\zeta$ be a Poisson process with a finite intensity measure $\bM$ on $\X$. Suppose that for $\bL$-a.e. $x \in \X$ we have coupled point processes $\xi^{x},\txi^{x}$ such that $\xi^{x}\stackrel{d}{=}\xi$, and $\txi^{x}\stackrel{d}{=}\xi^{!}_{x}$, respectively  ($\xi^{!}_{x}$ denotes the reduced Palm version of the point process $\xi$ at $x$). We also assume that $x\mapsto \mean{(\xi^{x} \triangle \txi^{x})(\X)}$ is measurable. Then 
\begin{align*}
\dkr(\xi,\zeta) & \leq d_{TV}(\bL,\bM) + 2 \int_{\X} \mean{(\xi^{x} \triangle \txi^{x})(\X)}\bL(\md x).
\end{align*}
\end{thm}

\begin{rem}
{\rm The measurability condition on $x\mapsto \mean{(\xi^{x} \triangle \txi^{x})(\X)}$ is satisfied if the couplings $\xi^{x}$ and $\txi^{x}$ are defined on the same probability space for all {$x$} and the random function $x\mapsto (\xi^{x}, \txi^{x})$ is measurable.}
\end{rem}

A version of the above theorem for the total variation distance was proven in \cite[Theorem 2.6]{Barbour1992} without the factor of $2$. The result of Barbour and Brown \cite{Barbour1992} was used in \cite[Theorem 3.3]{Otto2020} to derive total variation distance bounds for Poisson approximation of stabilizing statistics of Euclidean Poisson processes. The above theorem implies that the bounds in \cite[Theorem 3.3]{Otto2020} hold also under the stronger KR distance immediately, only with an additional factor of $2$. We shall later state a much more general theorem for stabilizing statistics in Section \ref{s:locdep}. Our proof of Theorem \ref{t:KR_coupling} is similar in spirit  to that of  \cite[Theorem 2.6]{Barbour1992},  using Lipschitz functions rather than bounded functions of the point processes.  
\begin{proof}[Proof of Theorem \ref{t:gen_Poisson_approx}]
Due to \eqref{e:dkrPoisson} we can assume that $\bM \equiv \bL$ throughout the proof.
Let $h\in \cLip(\X)$, and recall the definition of the generator $\cL$ of the Glauber dynamics in \eqref{d:gldyn}. Observe that $\cL h(\omega) \leq \bL(\X) + \omega(\X)$ for $\omega \in \widehat{\bN}_{\X}$. Using the finiteness of $\bL(\X)$ with Fubini-Tonelli Theorem, and the Campbell-Little-Mecke formula \eqref{e:clm}, we have that
\begin{align}
\mean{\cL h(\xi)} & = \int_{\X} \mean{D_xh(\xi)} \bL(\md x) - \int_{\X} \E^{!}_x\{D_xh(\xi)\}\bL(\md x) = \int_{\X} \left( \mean{D_xh(\xi)} - \mean{D_xh(\xi^{!}_x)} \right) \bL(\md x). \label{e:exp_lhphi_gen}
\end{align}
Using our coupling assumption and the Lipschitz assumption on $h$, we can bound the absolute value of the last integrand by
$$
 \mean{|D_xh(\xi^x) - D_xh(\txi^x)|} \leq \mean{|h(\xi^x + \delta_x) - h(\txi^x +\delta_x)| + |h(\xi^x) - h(\txi^x)|} \leq 2 \mean{(\xi^x \triangle \txi^x)(\X)}
$$
for $x\in\X$ so that
$$
|\mean{\cL h(\xi)}| \leq 2  \int_{\X} \mean{(\xi^x \triangle \txi^x)(\X)}  \bL(\md x) =:R_{\xi}.
$$
Similarly, replacing $h$ by $P_sh$ and using \eqref{e:pslip}, we have that
\[ |\mean{\cL P_sh(\xi)}| \leq e^{-s} R_{\xi}.\]
Integrating over $s$ yields
\begin{equation}
\label{e:Lpsh_bd}
\int_0^{\infty}  |\mean{\cL P_sh(\xi)}|ds \leq R_{\xi}.
\end{equation}
Combining \eqref{d:gldyn} with \eqref{e:pslip} leads to $|\cL P_sh(\xi)| \leq e^{-s}(\bL(\X) + \xi(\X))$. Using this and the finiteness of $\bL(\X)$, we can apply the Fubini-Tonelli Theorem in the Stein equation \eqref{e:stein} to conclude that
\eqb\label{e:int_bound} \mean{h(\zeta)} - \mean{h(\xi)} = \mean{\int_0^{\infty}\cL P_sh(\xi) \md s} = \int_0^{\infty}\mean{\cL P_sh(\xi)} \md s.
\eqe
Thus, from definition of KR distance \eqref{d:dkr}, \eqref{e:int_bound} and \eqref{e:Lpsh_bd} we have that
\begin{align*}
\dkr(\xi,\zeta)  = \sup_{h \in \cLip(\X)} \left|\int_0^{\infty}\mean{\cL P_sh(\xi)} \md s\right|
  \leq \sup_{h \in \cLip(\X)} \int_0^{\infty}\left |\mean{\cL P_sh(\xi)}\right | \md s \leq R_{\xi},
\end{align*}
as required. 
\end{proof}

Next, we want to apply the proof strategy of Theorem \ref{t:gen_Poisson_approx} in the specific case when $\xi$ is generated by statistics of a Poisson process. Let {$(\X,\cX)$ and $(\Y,\cY)$ be lcscH spaces} and let $f:\X^k \to \Y$ be a symmetric measurable function. Let $g : \X^k \times \bN_{\X} \to  \{0,1\}$ be another  measurable function that is also symmetric with respect to  the $\bx$ coordinates. For $\omega\in \bN_{\X}$, define
\begin{equation}
\label{e:defn_varphi}
\xi[\omega] := \frac{1}{k!}\sum_{\bx \in \omega^k_{\ne}}g(\bx,\omega)\delta_{f(\bx)}.
\end{equation}
In other words, the function $g$ is used as a ``flag"  indicating whether  or not to include a point at $f(\bx)$ in our point process $\xi[\omega]$.
We will study the case  $\xi :=\xi[\eta]$, where $\eta$ is a Poisson process on $\X$ with a $\sigma$-finite intensity measure $\bK$. A simple consequence of the Mecke formula \eqref{e:mmecke} is that the intensity measure $\bL$ of $\xi$ can be expressed  as
\begin{equation}
\label{e:int_L}
\bL(\md y) = \frac{1}{k!}\int_{\X^k} \ind\set{f(\bx) \in \md y} p(\bx) \bK^k(\md \bx),
\end{equation}
where $p(\bx) := \mean{g(\bx,\eta+\delta_{\bx})}$.
\begin{thm}
\label{t:KR_coupling}
Let $\eta$ be a Poisson process on $\X$ with a $\sigma$-finite intensity measure $\bK$. Let $\xi := \xi[\eta]$ be as defined in \eqref{e:defn_varphi} and such that $\bL(\Y) < \infty$. Suppose that for (almost) every $\bx\in\X^k$ with $p(\bx) > 0$, we have a pair of coupled point processes $\xi^{\bx}$ and $\txi^{\bx}$, such that $\xi^{\bx}\stackrel{d}{=}\xi$, and $\txi^{\bx} \stackrel{d}{=} (\xi[\eta + \delta_{\bx}] - \delta_{f(\bx)}) | \set{ g(\bx,\eta + \delta_\bx) = 1}$ (i.e., the conditional distribution of $\xi[\eta + \delta_{\bx}] - \delta_{f(\bx)}$ given $g(\bx,\eta + \delta_\bx) = 1$). Assume also that $\bx \mapsto\mean{(\xi^{\bx} \triangle \txi^{\bx})(\Y)}$ is measurable. Let $\zeta$ be a Poisson process on $\Y$ with a finite intensity measure $\bM$. Then, 
\begin{align*}
\dkr(\xi,\zeta) & \leq d_{TV}(\bL,\bM) +  \frac{2}{k!} \int_{\X^k} \mean{(\xi^{\bx} \triangle \txi^{\bx})(\Y)} p(\bx)\bK^k(\md \bx). 
\end{align*}
\end{thm}
\begin{rem}
\label{rem:KRcoupling}
{\rm Before proceeding to the proof, we show how the above theorem includes the results in \cite[Theorem 3.1]{penrose_inhomogeneous_2018} and \cite[Theorem 3.1]{decreusefond_functional_2016} as special cases. 
\begin{enumerate}
\item If $g(\bx,\omega) = g(\bx)$  we can set $dom(f) = \{ \bx : g(\bx) = 1\}$ and then $p(\bx) = \ind\set{\bx \in dom(f)}$. Further, there is a natural coupling where we take $\xi^{\bx} = \xi, \txi^{\bx} = \xi + \hat{\xi}[\bx,\eta + \delta_{\bx}]$, where $\hat{\xi}[\bx,\eta+ \delta_{\bx}]:= \xi[\eta + \delta_{\bx}] -\delta_{f(\bx)} - \xi[\eta]$ (see \cite[(6.2)]{decreusefond_functional_2016} for an explicit expression). By the assumptions on $f$ and $g$ we have that  $\hat{\xi}[\bx,\eta+ \delta_{\bx}] \in \widehat{\bN}_{\Y}$. In this case, we have that
$$ (\xi^{\bx} \triangle \txi^{\bx})(\Y) = \hat{\xi}[\bx,\eta+ \delta_{\bx}](\Y),$$
which is the same as the bound obtained for $R_{\eta}$ in \cite[page 2172]{decreusefond_functional_2016}. Following the arguments there,  we recover \cite[Theorem 3.1]{decreusefond_functional_2016}, which is one of the main theorems in that paper.

\item In many cases, the bound achieved using the coupling in Theorem \ref{t:KR_coupling}  is very similar to the bound one often obtains when applying \cite[Theorem 3.1] {penrose_inhomogeneous_2018}. This bound is merely the Poisson approximation for $\xi(\Y)$; see the proof of \cite[Theorem 2.8]{Iyer2020}, for example.  In some of these examples,  the bounds obtained for the approximation of the random \emph{variable} $\xi(\Y)$ can be extended to bounds for Poisson \emph{process} approximation with a few additional calculations.

\item The above theorem suffices in order to extend the results of \cite{Otto2020} (see Theorem 3.3 therein) under the KR distance. However, in Theorem \ref{t:qldep} we  will present an extension to more general Poisson driven point processes.
\end{enumerate} 
}
\end{rem}

Our proof combines ideas from the proof of Theorem \ref{t:gen_Poisson_approx} and those in the proof of \cite[Theorem 3.1]{penrose_inhomogeneous_2018}, for the case of $\xi(\Y)$. If we were to apply Theorem \ref{t:gen_Poisson_approx} directly, then we need a coupling of $\xi$ and its Palm measure on the space $\Y$.  This is not the coupling we have assumed in the above theorem. 
\begin{proof}[Proof of Theorem \ref{t:KR_coupling}]
Due to \eqref{e:dkrPoisson}, it suffices to prove the statement for the case $\bL \equiv \bM$.
Let $h\in\cLip(\Y)$. Recall that $\cL$ is the generator of the Glauber dynamics \eqref{d:gldyn} as in the proof of Theorem \ref{t:gen_Poisson_approx}. By \eqref{e:defn_varphi}, we obtain
\begin{align*}
\cL h(\xi) & = \int_{\Y} D_y h(\xi) \bL(\md y) - \int_{\Y} D_yh(\xi - \delta_y) \xi(\md y) \\
& =  \int_{\Y} D_y h(\xi) \bL(\md y) - { \frac{1}{k!} \sum_{\bx\in\eta_{\neq}^{k} } g(\bx,\eta)D_{f(\bx)}[h(\xi - \delta_{f(\bx)})].}
\end{align*}
Using Fubini-Tonelli, the finiteness of $\bL(\Y)$,  the multivariate Mecke formula \eqref{e:mmecke} and the definition of $\bL$ in \eqref{e:int_L}, we have that
\begin{equation}\label{e:exp_lhphi}
\begin{split}
\mean{\cL h(\xi)}  = \int_{\Y} \mean{D_yh(\xi)} \bL(\md y) - \frac{1}{k!}\int_{\X^k} \mean{g(\bx,\eta+  \delta_{\bx})D_{f(\bx)}[h(\xi[\eta +  \delta_{\bx}] - \delta_{f(\bx)})]} &\bK^k(\md \bx) \\
= \frac{1}{k!}  \int_{\X^k} \param{\mean{D_{f(\bx)}[h(\xi)]} - \cmean{D_{f(\bx)}[h(\xi[\eta +  \delta_{\bx}] - \delta_{f(\bx)})]}{ g(\bx,\eta+ \delta_{\bx}) = 1}} p(\bx) &\bK^k(\md \bx). 
\end{split}
\end{equation}
From the above identity, the existence of coupled point processes $\xi^{\bx}, \txi^{\bx}$ and the Lipschitz property \eqref{e.hlip_gen}, we can derive that
\begin{align*}
& |\mean{\cL h(\xi)}|  = \frac{1}{k!}\abs{\int_{\X^k} \param{ \mean{D_{f(\bx)}[h(\xi^{\bx})]} - \mean{D_{f(\bx)}[h(\txi^{\bx})] } } p(\bx) \bK^k(\md \bx)} \\
& \leq \frac{1}{k!} \int_{\X^k} \mean{\abs{h(\xi^{\bx} + \delta_{f(\bx)}) - h(\txi^{\bx} + \delta_{f(\bx)})}} p(\bx)\bK^k(\md \bx)  +  \frac{1}{k!} \int_{\X^k} \mean{\abs{h(\xi^{\bx}) - h(\txi^{\bx})}}  p(\bx)\bK^k(\md \bx) \\
& \leq \frac{2}{k!} \int_{\X^k} \mean{(\xi^{\bx} \triangle \txi^{\bx})(\Y)} p(\bx)\bK^k(\md \bx).
\end{align*}
Now, we can follow the proof of Theorem \ref{t:gen_Poisson_approx} by substituting the above bound instead of $R_{\xi}$ therein and complete the proof as in Theorem \ref{t:gen_Poisson_approx}.
\end{proof}

\section{Stabilization for Poisson input}
\label{s:locdep}

In this section, we consider the Poisson-driven point process $\xi = \xi[\eta]$, and extend the bound in Theorem \ref{t:KR_coupling} to the case where both $f$ and $g$ are locally dependent, i.e.,  may depend not only on $\bx$ but also on some (random) neighborhood around $\bx$. 

{Recall that $(\X,\cX)$ and $(\Y,\cY)$ are lcscH spaces.} Let $g : \X^k \times \bN_{\X} \to  \{0,1\}, f : \X^k \times \bN_{\X} \to \Y$ be measurable functions that are also symmetric in the $\bx$ coordinates.  Let $\cF := \cF(\X)$ be the space of all closed subsets of $\X$ equipped with the smallest $\sigma$-algebra containing $\{F \in \cF : F \cap K \neq \emptyset\}$ for all compact $K \subset \X$.  This $\sigma$-algebra is the Borel $\sigma$-algebra generated by the Fell topology.  Let $\cS : \X^{k} \times \bN_{\X} \to \cF$ be a measurable function. In this section we assume that $f$ and $g$ are \emph{localized} to $\cS$, i.e., for all $\omega \in \bN_{\X}$, for  all $\bx \in \omega^k_{\ne}$ and for all $S \supset \cS(\bx,\omega)$, we have that
\begin{equation}\label{e:AssumptionRi}
\begin{split}
g(\bx,\omega)&=g(\bx,\omega \cap S),\quad\text{and} \\
 f(\bx,\omega)&=f(\bx,\omega\cap S), \quad \text{if} \quad g(\bx,\omega)=1,
\end{split}
\end{equation}
where we  use $\omega \cap S$ as a multiset. We also assume that $\cS({\bx},\omega)$ is a stopping set, i.e., for every compact set $S\subset \X$ we have that
$$ \set{\omega \in \bN_{\X} : \cS({\bx},\omega)\subset S} =  \set{ \omega \in \bN_{\X} : \cS(\bx, \omega \cap S) \subset S}.$$
The construction of a measurable $\cS$ as above is not always obvious.  Often, such an $\cS$ is {constructed as a ball of random radius with the radius being called as  {\em radius of stabilization}}. The justification for the measurability of the radius of stabilization {in the Euclidean case} can be found in \cite[Definition 2.1 and below]{penrose_lln_2007}.  {We also refer the reader to \cite[Section 2]{Otto2020} for more information on stopping sets (see also \cite[Appendix A]{BaumstarkLast09} and \cite[Appendix A]{Last2021}).  Finally, we point out that more general regions than balls can occur as `localizing regions' $\cS$ (see \cite{Bhattacharjee2021}). }

Next, define
\begin{equation}
\label{e:defn_xi_ld}
\xi[\omega] := \frac{1}{k!}\sum_{\bx \in \omega^k_{\ne}}g(\bx,\omega)\delta_{f(\bx,\omega)}, 
\end{equation}
and consider $\xi:=\xi[\eta]$ where $\eta$ is a Poisson process with {a $\sigma$-finite} intensity measure $\bK$, as before. Using the multivariate Mecke formulae \eqref{e:mmecke}, we have that the intensity measure $\bL$ of $\xi$ is given by
\begin{equation}
\label{e:int_L_ld}
\bL(\md y) = \frac{1}{k!}\int_{\X^k} \mean{\ind\set{f(\bx,\eta+\delta_{\bx}) \in \md y} g(\bx,\eta+\delta_{\bx})} \bK^k(\md \bx).
\end{equation}
For a set $A\subset \X$ and $\bx=(x_1,\hdots,x_k)\in\X^k$ we write $\bx\subset A$ to denote $\{x_1,\hdots,x_k\}\subset A$.

\begin{thm}
\label{t:qldep}
Let $\xi = \xi[\eta]$ be the point process defined as in \eqref{e:defn_xi_ld} with $f,g$ satisfying \eqref{e:AssumptionRi} {and $\bL(\Y)<\infty$.} Let $\zeta$ be a Poisson process with {a finite} intensity measure $\bM$. Further, suppose that we are given a measurable mapping $\bx \mapsto S_{\bx}$ from $\X^k$ to $\cF$ satisfying {$\bx \subset S_{\bx}$.} Define 
\[
\tg(\bx,\omega) := g(\bx,\omega)\ind\set{\cS(\bx,\omega) \subset S_{\bx}}
\] 
{for $\omega \in \bN_{\X}$ }, and
\[
{\xi_{\mathrm{tr}} :=\xi_{\mathrm{tr}}[\eta] := \frac{1}{k!}\sum_{\bx \in \eta^k_{\ne}}\tilde g(\bx,\eta)\delta_{f(\bx,\eta)}.}
\]
Then
\begin{align*}
\dkr(\xi,\zeta)  &\le d_{TV}(\bL,\bM) + 2\param{\varx{\xi_{\mathrm{tr}}(\Y)} - \meanx{\xi_{\mathrm{tr}}(\Y)} } + E_1 + 2E_2\\
&=  d_{TV}(\bL,\bM) + E_1 + E_2 + E_3 + E_4
\end{align*}
with
\[
\begin{split}
E_1 & {:=} \frac{2}{k!} \int_{\X^k} \mean{g(\bx,\eta+\delta_{\bx}) \ind\set{\cS(\bx,\eta+\delta_{\bx})\not\subset S_{\bx}}} \bK^k(\md\bx),\\ 
E_2 & {:=} \frac{2}{(k!)^2}  \int_{\X^{k}} \int_{\X^{k}} \ind\set{S_{\bx}\cap S_{\bz}\ne \emptyset} \mean{\tg(\bx,\eta+\delta_{\bx})} \mean{\tg(\bz,\eta+\delta_{\bz})} \bK^k(\md \bz)\bK^k(\md \bx),\\
E_3 & {:=} \frac{2}{(k!)^2}  \int_{\X^{k}} \int_{\X^{k}} \ind\set{S_{\bx}\cap S_{\bz}\ne \emptyset} \mean{\tg(\bx,\eta+\delta_{\bx}+\delta_{\bz})\tg(\bz,\eta+\delta_{\bx}+\delta_{\bz})}  \bK^k(\md\bz)\bK^k(\md\bx) ,\\
E_4 & {:=} \frac{2}{k!}\sum_{ \emptyset\subsetneq I \subsetneq {\{1,\hdots,k\}}} \frac{1}{(k-|I|)!} \int_{\X^k}\int_{\X^{k-|I|}}\mean{ \tg(\bx, \eta+\delta_{\bx}+\delta_{\bz}) \tg((\bx_{I},\bz),\eta+\delta_{\bx}+\delta_{\bz})}   \bK^{k-|I|}(\md\bz) \bK^k(\md\bx),
\end{split}
\]
{where for $I = \{i_1,\ldots, i_m\}$ we set $\bx_I = (x_{i_1},\ldots, x_{i_m})$, and $(\bx_I,\bz) = (x_{i_1},\ldots, x_{i_m}, z_1,\ldots, z_{k-m})$.}
\end{thm}
Note that the bound is similar to that in \cite[Theorem 3.3]{Otto2020} while it holds for a more general class of point processes  under the stronger KR distance {and  for a {more} general notion of `localizing region' than balls}.  The proof of  \cite[Theorem 3.3]{Otto2020} proceeds by constructing a coupling between $\xi$ and $\xi^{!}_{x}$ and then using \cite[Theorem 2.6]{Barbour1992} (the total variation distance analogue of our Theorem \ref{t:gen_Poisson_approx}). The proof here is much simpler -- we use the canonical truncation via $S_{\bx}$ to reduce the proof to Poisson approximation bounds for truncated point processes with bounded range of dependence.  We first prove the bound for the truncated point process and then use this result to  complete the proof of Theorem \ref{t:qldep}.

\begin{proof}
{As before, we assume that $\bL=\bM$, and use \eqref{e:dkrPoisson} otherwise.   The proof will proceed in two steps as described above.}

\noindent {\bf Step 1 : The case of bounded range of dependence.}
We first assume that for all $\bx\in \X^k$ and $\omega\in \bN_{\X}$ we have that $\cS(\bx,\omega) \subset S_{\bx}$. {This implies that $\tilde{g}=g$, $\xi = \xi_{\mathrm{tr}}$ and $E_1=0$.} {Moreover, \eqref{e:AssumptionRi} leads to}
\begin{equation}\label{e:AssumptionRk}
g(\bx,\omega)=g(\bx,\omega\cap S_{\bx}) \quad \text{and} \quad {\text{if} \quad g(\bx,\omega)=1,} \quad f(\bx,\omega)=f(\bx,\omega\cap S_{\bx}).
\end{equation}
Fix $h\in\cLip (\Y)$. Similarly to the derivation in \eqref{e:exp_lhphi}, we use \eqref{e:int_L_ld} to obtain that  
\[
\splitb
\mean{\cL h(\xi)}  &=\frac{1}{k!}  \int_{\X^k} \mean{g(\bx,\eta'+  \delta_{\bx})D_{f(\bx,\eta'+\delta_\bx)}[h(\xi[\eta])]} \bK^k(\md \bx)\\
& \quad - \frac{1}{k!}\int_{\X^k} \mean{g(\bx,\eta+  \delta_{\bx})D_{f(\bx,\eta+\delta_\bx)}\sbrk{h(\xi[\eta +  \delta_{\bx}] - \delta_{f(\bx,\eta+\delta_{\bx})})}} \bK^k(\md \bx),
\splite
\]
where $\eta'$ is a Poisson process, independent of $\eta$ and $\eta'\stackrel{d}{=} \eta$.
Integrating over $s$ as in \eqref{e:int_bound}, yields
\begin{equation}\label{e:stein_local}
\splitb
\mean{h(\zeta)} - \mean{h(\xi)}
& =\frac{1}{k!} \int_0^\infty \int_{\X^k} \left(\mean{g(\bx,\eta'+\delta_{\bx}) D_{f(\bx,\eta'+\delta_{\bx})}[P_sh(\xi[\eta])] }\right. \\
& \left.\quad - \mean{g(\bx,\eta+\delta_{\bx}) D_{f(\bx,\eta+\delta_{\bx})}[P_sh(\xi[\eta+\delta_{\bx}]-\delta_{f(\bx,\eta+\delta_{\bx})})] }\right) \bK^k(\md\bx) \md s.
\splite
\end{equation}
For $\bx\in\X^k$  define
\begin{equation}
\label{e:xix}
\splitb
\xi_{\bx} := \xi_{\bx}[\eta] &= \frac{1}{k!}\sum_{\bz\in \eta_{\neq}^k: S_{\bz}\cap S_{\bx} = \emptyset }g(\bz,\eta) \delta_{f(\bz,\eta)} = \frac{1}{k!}\sum_{\bz\in \eta_{\neq}^k: S_{\bz}\cap S_{\bx} = \emptyset }g(\bz,\eta+\delta_{\bx}) \delta_{f(\bz,\eta+\delta_{\bx})} \\
&= \frac{1}{k!}\sum_{\bz \in (\eta \cap  S_{\bx}^c)_{\neq}^k: S_{\bz}\cap S_{\bx} = \emptyset }g(\bz,\eta \cap S_{\bx}^c) \delta_{f(\bz,\eta \cap S_{\bx}^c)},
\splite
\end{equation}
where the second and third equalities follow from the assumption \eqref{e:AssumptionRk} on $f$ and $g$, and the fact that ${\bz} \subset S_{\bz}\subset S_{\bx}^c$. Now, it follows from \eqref{e:pslip} that 
\begin{align}
& \big| \mean{g(\bx,\eta'+\delta_{\bx}) D_{f(\bx,\eta'+\delta_{\bx})}[P_sh(\xi[\eta])] } - \mean{g(\bx,\eta'+\delta_{\bx}) D_{f(\bx,\eta'+\delta_{\bx})}[P_sh(\xi_{\bx}[\eta])] } \big| \no  \\
& \leq \frac{2}{k!}e^{-s} \mean{ g(\bx,\eta'+\delta_{\bx}) \sum_{\bz\in \eta^k_{\ne} : S_{\bz}\cap S_{\bx}\ne \emptyset} g(\bz,\eta)} \no \\
\label{e:localdepbd1} & \leq \frac{2}{k!}e^{-s} \int_{\X^k}\ind\set{S_{\bz}\cap S_{\bx}\ne \emptyset} \mean{g(\bx,\eta+\delta_{\bx})} \mean{g(\bz,\eta+\delta_{\bz})} \bK^k(\md\bz).
\end{align}
Observe that by definition \eqref{e:xix} we have $\xi_{\bx}[\eta] \subset \xi[\eta+\delta_{\bx}]-g(\bx,\eta+\delta_{\bx})\delta_{f(\bx,\eta+\delta_{\bx})}$ as multisets. Thus, we can define the point process
$$\hat\xi_{\bx} := \hat\xi_{\bx}[\eta] :=  \xi[\eta+\delta_{\bx}]-g(\bx,\eta+\delta_{\bx})\delta_{f(\bx,\eta+\delta_{\bx})} - \frac{1}{k!}\sum_{\bz\in \eta^k_{\ne}} g(\bz,\eta+\delta_{\bx})\delta_{f(\bz,\eta+\delta_{\bx})}.$$
In other words, $\hat\xi_{\bx}[\eta]$ contains the {points of} $\xi[\eta+\delta_{\bx}]$ that are generated by nonempty strict  subsets of $\bx$. 
Using \eqref{e:pslip} and {the} Mecke formula, we derive that
\begin{align}
& \big|\mean{g(\bx,\eta+\delta_{\bx}) D_{f(\bx,\eta+\delta_{\bx})}[P_sh(\xi[\eta+\delta_{\bx}]-\delta_{f(\bx,\eta+\delta_{\bx})})] } 
- \mean{g(\bx,\eta+\delta_{\bx}) D_{f(\bx,\eta+\delta_{\bx})}[P_sh(\xi_{\bx}[\eta])] } \big| \no  \\
& \le 2e^{-s} \param{\mean{ g(\bx,\eta+\delta_{\bx}) \frac{1}{k!}\sum_{\bz\in \eta^k_{\ne}:S_{\bz}\cap S_{\bx}\ne \emptyset} g(\bz,\eta+\delta_{\bx})} + \mean{g(\bx,\eta+\delta_{\bx})\hat\xi_{\bx}[\eta](\Y)}} \no \\
& \leq \frac{2}{k!}e^{-s} \int_{\X^k}\ind\set{S_{\bz}\cap S_{\bx}\ne \emptyset}\mean{g(\bx,\eta+\delta_{\bx}+\delta_{\bz}) g(\bz,\eta+\delta_{\bx}+\delta_{\bz}) } \bK^k(\md\bz) \no \\
\label{e:localdepbd2}  & \quad +
 2e^{-s}\mean{g(\bx,\eta+\delta_{\bx})\hat\xi_{\bx}[\eta](\Y)}.
\end{align}
By assumption \eqref{e:AssumptionRk}, $f(\bx,\eta+\delta_{\bx})$ and $g(\bx,\eta+\delta_{\bx})$ depend only on $\eta\cap S_{\bx}$, while $\xi_{\bx}$ is a functional of $\eta\cap S_{\bx}^c$ (see \eqref{e:xix}). Together with the independence property of Poisson processes, we obtain that
\begin{equation}
\label{e:localdepbd3}
\mean{g(\bx,\eta'+\delta_{\bx}) D_{f(\bx,\eta'+\delta_{\bx})}P_sh(\xi_{\bx}[\eta]) } = \mean{g(\bx,\eta+\delta_{\bx}) D_{f(\bx,\eta+\delta_{\bx})}P_sh(\xi_{\bx}[\eta]) }.
\end{equation}
Substituting \eqref{e:localdepbd1}, \eqref{e:localdepbd2} and \eqref{e:localdepbd3} into \eqref{e:stein_local}, using the triangle inequality and integrating over $s$, yields
\begin{align*}
& \big|\mean{h(\zeta)} - \mean{h(\xi)}\big| \le \frac{2}{(k!)^2}\int_{\X^k} \int_{\X^k}\ind\set{S_{\bz}\cap S_{\bx}\ne \emptyset} \mean{g(\bx,\eta+\delta_{\bx})} \mean{g(\bz,\eta+\delta_{\bz})} \bK^k(\md\bz) \bK^k(\md\bx) \\
&\  \ \quad + \frac{2}{(k!)^2}\int_{\X^k} \int_{\X^k}\ind\set{S_{\bz}\cap S_{\bx}\ne \emptyset} \mean{g(\bx,\eta+\delta_{\bx}+\delta_{\bz}) g(\bz,\eta+\delta_{\bx}+\delta_{\bz}) } \bK^k(\md\bz) \bK^k(\md\bx)\\
&\  \ \quad + \frac{2}{k!} \int_{\X^k}   \mean{ g(\bx,\eta+\delta_{\bx})\hat\xi_{\bx}[\eta](\Y)} \bK^k(\md\bx).
\end{align*}
Here, the first and the second term on the right-hand side are $E_2$ and $E_3$, respectively.
Notice that 
\[
\splitb
	  &\frac{2}{k!}\int_{\X^k}\mean{ g(\bx,\eta+\delta_{\bx})\hat\xi_{\bx}[\eta](\Y)} \bK^k(\md\bx)\\
	&\ = \frac{2}{k!}\sum_{ \emptyset\subsetneq I \subsetneq {\{1,\hdots,k\}}} \frac{1}{(k-|I|)!}
	\int_{\X^k}\int_{\X^{k-|I|}}\mean{ g(\bx, \eta+\delta_{\bx}+\delta_{\bz}) g((\bx_{I},\bz),\eta+\delta_{\bx}+\delta_{\bz})} \bK^{k-|I|}(\md\bz) \bK^k(\md\bx),
\splite
\]
which is equal to $E_4$.
Together with $E_1=0$, this proves the second bound in the theorem. A short computation shows that
\begin{align*}
E_4 = 2\mean{\xi(\Y)^2} - 2\mean{\xi(\Y)} - 2\param{\frac{1}{k!}}^2 \int_{\X^k}\int_{\X^{k}}\mean{ g(\bx, \eta+\delta_{\bx}+\delta_{\bz}) g(\bz,\eta+\delta_{\bx}+\delta_{\bz})} \bK^{k}(\md\bz) \bK^k(\md\bx).
\end{align*}
From \eqref{e:AssumptionRk} and the independence property of $\eta$ it follows that
\begin{align*}
& 2 \param{\frac{1}{k!}}^2 \int_{\X^k}\int_{\X^{k}}\mean{ g(\bx, \eta+\delta_{\bx}+\delta_{\bz}) g(\bz,\eta+\delta_{\bx}+\delta_{\bz})} \bK^{k}(\md\bz) \bK^k(\md\bx) \\
& = 2 \mean{\xi(\Y)}^2 + 2 \param{\frac{1}{k!}}^2 \int_{\X^k}\int_{\X^{k}}\mean{ g(\bx, \eta+\delta_{\bx}+\delta_{\bz}) g(\bz,\eta+\delta_{\bx}+\delta_{\bz})} \\
& \qquad \qquad \qquad \qquad \qquad \qquad \qquad - \mean{ g(\bx, \eta+\delta_{\bx}+\delta_{\bz})} \mean{g(\bz,\eta+\delta_{\bx}+\delta_{\bz})} \bK^k(\md\bz)  \bK^k(\md\bx) \\
& = 2 \mean{\xi(\Y)}^2 + 2 \param{\frac{1}{k!}}^2 \int_{\X^k}\int_{\X^{k}} \ind\set{S_{\bx}\cap S_{\by}\neq\emptyset }  \big( \mean{ g(\bx, \eta+\delta_{\bx}+\delta_{\bz}) g(\bz,\eta+\delta_{\bx}+\delta_{\bz})} \\
& \qquad   \qquad \qquad \qquad \qquad \qquad \qquad -  \mean{ g(\bx, \eta+\delta_{\bx}+\delta_{\bz})} \mean{g(\bz,\eta+\delta_{\bx}+\delta_{\bz})} \big) \bK^k(\md\bz) \bK^k(\md\bx) \\
& = 2 \mean{\xi(\Y)}^2 + E_3- E_2.
\end{align*}
Combining the previous identities yields that
$$
E_4 = 2\param{\varx{\xi(\Y)} - \meanx{\xi(\Y)} } + E_2 -E_3,
$$
which proves the first bound in the theorem.

{\bf Step 2 : The general case.} We will use Step 1 to complete the proof of the general result in the {theorem.}

{We denote by $\bL_{\mathrm{tr}}$ the intensity measure of the truncated point process $\xi_{\mathrm{tr}}$, which was defined in the statement of the theorem. It follows from the multivariate Mecke equation \eqref{e:mmecke} that
$$
\bL_{\mathrm{tr}}(\md y) = \frac{1}{k!}\int_{\X^k} \mean{g(\bx,\eta+\delta_{\bx}) \ind\set{\cS(\bx,\eta+\delta_{\bx}) \subset S_{\bx}} \ind\set{f(\bx,\eta+\delta_{\bx})\in \md y}} \bK^k(\md\bx).
$$
Let $\zeta_{\mathrm{tr}}$ be a Poisson process with intensity measure $\bL_{\mathrm{tr}}$. From the triangle inequality for the KR distance along with \eqref{e:dtvbd} and \eqref{e:dkrPoisson0} we obtain
\begin{align}
\dkr(\xi,\zeta) \leq \dkr(\xi,\xi_{\mathrm{tr}}) +  \dkr(\xi_{\mathrm{tr}},\zeta_{\mathrm{tr}}) + \dkr(\zeta,\zeta_{\mathrm{tr}}) \leq \mean{d_{TV}(\xi,\xi_{\mathrm{tr}})} +   \dkr(\xi_{\mathrm{tr}},\zeta_{\mathrm{tr}}) + d_{TV}(\bL,\bL_{\mathrm{tr}}) \no.
\end{align}
Since  $\bL_{\mathrm{tr}}(A) \leq \bL(A)$ and $\xi_{\mathrm{tr}}(A) \leq \xi(A)$ for all $A\in\cY$, and using \eqref{e:int_L_ld}, we have
$$
d_{TV}(\bL,\bL_{\mathrm{tr}}) = \bL(\Y) - \bL_{\mathrm{tr}}(\Y) = \frac{1}{k!}\int_{\X^k} \mean{g(\bx,\eta+\delta_{\bx}) \ind\set{\cS(\bx,\eta+\delta_{\bx}) \not\subset S_{\bx}}}  \bK^k(\md\bx) = \frac{E_1}{2},
$$
and
$$
\mean{d_{TV}(\xi,\xi_{\mathrm{tr}})} = \mean{ \xi(\Y) - \xi_{\mathrm{tr}}(\Y) } = \bL(\Y) - \bL_{\mathrm{tr}}(\Y) = \frac{E_1}{2}.
$$
Thus, we have shown that
\begin{equation}\label{e:hbd_trunc}
\dkr(\xi,\zeta) \leq \dkr(\xi_{\mathrm{tr}},\zeta_{\mathrm{tr}}) + E_1.
\end{equation}}
Recall that $\tg(\bx,\eta+\delta_{\bx}) := g(\bx,\eta+\delta_{\bx})\ind\set{\cS(\bx,\eta+\delta_{\bx}) \subset S_{\bx}} $, and observe that by the stopping set property of $\cS$, both {$f$ and $\tg$ satisfy} the local dependence assumptions in \eqref{e:AssumptionRk}. Hence,  using Step 1 of the proof,  we can derive the necessary bounds for $\dkr(\xi_{\mathrm{tr}},\zeta_{\mathrm{tr}})$ in \eqref{e:hbd_trunc} and complete the proof.
\end{proof}

\section{Stabilization for binomial input}
\label{s:locdep_binomial}

 In this section we focus on the binomial {point} process $\beta_n$  (replacing the Poisson process), i.e., a point process consisting of $n$ independent points, distributed according to some probability measure $\mathbf{Q}$. In this case, we have that for any measurable function {$h: \X^k \times \bN_{\X}\to\R_+$} and $n\geq k$,
\begin{equation}\label{e:IdentityBinomial}
{\mean{\sum_{\bx \in \beta^k_{n,\ne}} h(\bx,\beta_n)} = (n)_k \int_{\X^k}\mean{h(\bx,\beta_{n-k}+\delta_{\bx})}\mathbf{Q}^k(\md \bx),}
\end{equation}
with $(n)_k:=n\cdot\hdots\cdot (n-k+1)$. Note that \eqref{e:IdentityBinomial} is the analogue for binomial point processes of the multivariate Mecke formula \eqref{e:mmecke}.

Let $\beta_n$ be a binomial point process, and let
$$
\xi:= { \xi[\beta_n] :=} \sum_{x\in\beta_n} g(x,\beta_n) \delta_{f(x,\beta_n)}.
$$
{Here, $g$ and $f$ are as in Section \ref{s:locdep} for an underlying Poisson process, but we allow only $k=1$ in the sequel.} By \eqref{e:IdentityBinomial}, the point process $\xi$ has an intensity measure
\begin{equation}\label{e:IntensityMeasureBinomial}
\bL( {\md y})= n \int_{\X} \mean{g(x,\beta_{n-1}+\delta_x) \ind\{f(x,\beta_{n-1}+\delta_x)\in \md y\}} \mathbf{Q}(\md x).
\end{equation}

\begin{thm}\label{thm:binomial}
Let $\xi$ and $\bL$ be as above with $f,g$ satisfying \eqref{e:AssumptionRi} {with a measurable function $\cS: \X\times\bN_{\X}\to\cF$} and let $\zeta$ be a Poisson process with a finite intensity measure $\bM$. Further, suppose that we are given a measurable mapping $x \mapsto S_{x}$ from $\X$ to $\cF$ satisfying $x \in S_{x}$ and $\mathbf{Q}(S_x)<1$ for all $x\in\X$.    For $\omega \in \bN_{\X}$ and $x \in \omega$, define 
\[
\tg(x,\omega) := g(x,\omega)\ind\set{\cS(x,\omega) \subset S_{x}}.
\] 
Then
\begin{align*}
\dkr(\xi,\zeta) & \leq d_{TV}(\bL,\bM)  + E_1 + E_2 + E_3 + E_4 + E_5 + E_6,
\end{align*}
where
\begin{align*}
E_1 &:= 2n\int_{\X} \mean{g(x,\beta_{n-1}+\delta_{x})\ind\set{\cS(x,\beta_{n-1}+\delta_{x})\not \subset S_{x}}}\mathbf{Q}(\md x),\\
E_2 &:= 2n^2 \int_{\mathbb{X}^2} \ind\{S_x\cap S_y \neq \emptyset \} \mean{\tg(x,\beta_{n-1}+\delta_x)} \mean{\tg(y,\beta_{n-1}+\delta_y)} \mathbf{Q}^2(\md(x,y)),  \\
E_3 &:= 2n^2 \int_{\mathbb{X}^2} \ind\{S_x\cap S_y \neq \emptyset \} \mean{\tg(x,\beta_{n-2}+\delta_x+\delta_y) \tg(y,\beta_{n-2}+\delta_x+\delta_y)} \mathbf{Q}^2(\md(x,y)), 
\end{align*}
\begin{align*}
E_4 &:= 2n\int_{\mathbb{X}} \bigg( (1+n\mathbf{Q}(S_x)) \mean{\tg(x,\beta_{n-1}+\delta_x)} + n \int_{S_x} \mean{\tg(x,\beta_{n-2}+\delta_x+\delta_z)} \mathbf{Q}(\md z)  \bigg)\\
& \qquad \quad \times \bigg( \int_{\mathbb{X}} \ind\{S_x\cap S_y=\emptyset\} \mean{\tg(y,\widetilde{\beta}_{x,n-1}+\delta_y)} \mathbf{Q}_x(\md y)\\
& \qquad \qquad  + n \int_{\mathbb{X}^2} \ind\{S_x\cap S_{y_1}=\emptyset, y_2\in S_{y_1}\} \mean{\tg(y_1,\widetilde{\beta}_{x,n-2}+\delta_{y_1}+\delta_{y_2}} \mathbf{Q}^2_x(\md(y_1,y_2)) \bigg) \mathbf{Q}(\md x),\\
E_5 &:= 2n^3 \int_{\mathbb{X}^2} \ind\{S_x\cap S_y=\emptyset\} \mean{\tg(x, \beta_{n-2}+\delta_x) } \mean{ \tg(y, \beta_{n-2}+\delta_y) } \mathbf{Q}(S_x) \mathbf{Q}_x(S_y) \mathbf{Q}^2(\md(x,y)), \\
E_6 &:= { 2n^2 \int_{\mathbb{X}^2} \ind\{S_x\cap S_y=\emptyset\} \mean{ \tg(x,\beta_{n-2}+\delta_x) \tg(y,\beta_{n-2}+\delta_y) } \mathbf{Q}_x(S_y) \mathbf{Q}^2(\md(x,y)) } \\
& \quad \ { + 2n^3 \int_{\mathbb{X}^2} \int_{S_x} \ind\{S_x\cap S_y=\emptyset\} \mean{ \tg(x,\beta_{n-3}+\delta_x+\delta_z) \tg(y,\beta_{n-3}+\delta_y) } \mathbf{Q}_x(S_y)  \mathbf{Q}(\md z) \mathbf{Q}^2(\md(x,y)) , }
\end{align*}
and where  $\widetilde{\beta}_{x,m}$ is a binomial point process of $m$ independent points distributed according to $\mathbf{Q}_x(\cdot):=\mathbf{Q}(\cdot\cap S_x^c)/\mathbf{Q}(S_x^c)$.
\end{thm}
\begin{rem}
\label{r:krbin}
A few remarks are in place regarding the above theorem in relation to the approximation {results} in the previous sections. 
{\rm
\begin{enumerate}
\item The binomial point process does not possess the same spatial independence as the Poisson process and hence the bounds are more complicated here. However, we expect that in many applications, the computation of the bounds should not be more difficult than in the Poisson case.

\item Since the bounds are already quite involved, we have restricted ourselves to only the case $k=1$, i.e., to functions $f,g$ that take $x \in \X$ as an input instead of $\bx \in \X^k$. 

\item Our proof shall follow the same strategy as in Theorem \ref{t:qldep} but accounting for the  added complication due to the lack of spatial independence in the binomial point process.  In the special case  where $f(x,\omega) = f(x)$ (as in Theorem \ref{t:KR_coupling}) we can use a coupling approach as in the proof of Theorem \ref{t:KR_coupling}, but this shall not significantly simplify the proof or the bounds in the above theorem.
\end{enumerate}
}
\end{rem}
\begin{proof}
By \eqref{e:dkrPoisson}, we can assume as before that $\bL=\bM$.  Further,  we shall prove the theorem under the assumption that $\tg \equiv g$ and follow the arguments as in Step 2 in the proof of Theorem \ref{t:qldep}, to complete the proof of the general case.  Observe that in the case of $\tg \equiv g$ we have $E_1 = 0$. 

Assume from here onwards that $\tg \equiv g$,  {i.e.,  $\cS(x,\omega) \subset S_{x}$} for all $x \in \X$ and {$\omega\in\bN_{\X}$.} For a Lipschitz function $h$, we have by \eqref{d:gldyn} and \eqref{e:stein} that
$$
\mean{h(\zeta)} - \mean{h(\xi)} = \int_0^\infty \int_{\Y} \mean{D_yP_sh(\xi)} \bL(\md y) \md s - \int_0^\infty \mean{\sum_{y\in\xi} D_yP_sh(\xi-\delta_y)} \md s.
$$
{Using \eqref{e:IdentityBinomial}} we have
\begin{align*}
& \int_0^\infty \mean{\sum_{y\in\xi} D_yP_sh(\xi-\delta_y)} \md s = \int_0^\infty \mean{\sum_{x\in\beta_n} g(x,\beta_n) D_{f(x,\beta_n)}P_sh(\xi-\delta_{f(x,\beta_n)})} \md s \\
& = n \int_0^\infty \int_{\X} \mean{ g(x,\beta_{n-1}+\delta_x) D_{f(x,\beta_{n-1}+\delta_x)}P_sh(\xi[\beta_{n-1}+\delta_x]-\delta_{f(x,\beta_{n-1}+\delta_x)})} \mathbf{Q}(\md x) \md s.
\end{align*}
{Denoting by $\beta_m'$ an independent copy of $\beta_m$,} we obtain from \eqref{e:IntensityMeasureBinomial} that
$$
\int_0^\infty \int_{\Y} \mean{D_yP_sh(\xi)} \bL(\md y) \md s = n \int_0^\infty \int_{\X} \mean{g(x,\beta_{n-1}'+\delta_x) D_{f(x,\beta_{n-1}'+\delta_x)}P_sh(\xi[\beta_n]) } \mathbf{Q}(\md x) \md s.
$$
Combining the three previous identities leads to
\begin{align*}
 \mean{h(\zeta)} - \mean{h(\xi)} &= n \int_0^\infty \int_{\X} \mean{g(x,\beta_{n-1}'+\delta_x) D_{f(x,\beta_{n-1}'+\delta_x)}P_sh(\xi[\beta_n]) } \\
& \hspace{1.25cm} - \mean{g(x,\beta_{n-1}+\delta_x) D_{f(x,\beta_{n-1}+\delta_x)}P_sh(\xi[\beta_{n-1}+\delta_x]-\delta_{f(x,\beta_{n-1}+\delta_x)}) } \mathbf{Q}(\md x) \md s.
\end{align*}
For $x\in\X$ and $\omega\in\mathbf{N}$ we define
$$
\xi_{x}[\omega]= \sum_{ y\in \omega, S_x\cap S_y=\emptyset}g(y,\omega) \delta_{f(y,\omega)}.
$$
For $x\in\X$ and {$s\geq 0$} let
\begin{align*}
T_{1,x,s} & := \big| \mean{g(x,\beta_{n-1}'+\delta_x) D_{f(x,\beta_{n-1}'+\delta_x)}P_sh(\xi[\beta_n]) } - \mean{g(x,\beta_{n-1}'+\delta_x) D_{f(x,\beta_{n-1}'+\delta_x)}P_sh(\xi_x[\beta_{n}]) } \big|, \\
T_{2,x,s} & := \big|\mean{g(x,\beta_{n-1}+\delta_x) D_{f(x,\beta_{n-1}+\delta_x)}P_sh(\xi[\beta_{n-1}+\delta_x] -\delta_{f(x,\beta_{n-1}+\delta_x)}) } \\
& \quad \quad 
- \mean{g(x,\beta_{n-1}+\delta_x) D_{f(x,\beta_{n-1}+\delta_x)}P_sh(\xi_x[\beta_{n-1}]) } \big|, \\
T_{3,x,s} & = \big|  \mean{g(x,\beta_{n-1}'+\delta_x) D_{f(x,\beta_{n-1}'+\delta_x)}P_sh[\xi_x(\beta_{n}]) }  
- \mean{g(x,\beta_{n-1}+\delta_x) D_{f(x,\beta_{n-1}+\delta_x)}P_sh(\xi_x[\beta_{n-1}]) } \big|
\end{align*}
so that by the triangle inequality
\begin{equation}\label{e:InequalityTs}
\big|  \mean{h(\zeta)} - \mean{h(\xi)} \big| \leq n \int_0^\infty \int_{\X} \left( T_{1,x,s} + T_{2,x,s} + T_{3,x,s} \right) \, \mathbf{Q}(\md x) \md s.
\end{equation}
Fix $x\in\X$ and $s\geq 0$. It follows from \eqref{e:pslip} that
\begin{align*}
T_{1,x,s} 
& \leq 2e^{-s} \mean{ g(x,\beta_{n-1}'+\delta_x) d_{TV}(\xi[\beta_n],\xi_x[\beta_{n}]) }  = 2e^{-s} \mean{ g(x,\beta_{n-1}'+\delta_x)} \mean{d_{TV}(\xi[\beta_n],\xi_x[\beta_{n}]) },
\end{align*}
and
\begin{align*}
T_{2,x,s} 
& \leq 2e^{-s} \mean{ g(x,\beta_{n-1}+\delta_x) d_{TV}(\xi[\beta_{n-1}+\delta_x] -\delta_{f(x,\beta_{n-1}+\delta_x)},\xi_x[\beta_{n-1}]) }. 
\end{align*}
By \eqref{e:IdentityBinomial} we have
\begin{align*}
\mean{d_{TV}(\xi[\beta_{n}],\xi_x[\beta_{n}])} & \leq \mean{ \sum_{y\in\beta_{n}, S_x\cap S_y\neq\emptyset} g(y,\beta_{n}) } = n \int_{\mathbb{X}} \ind\{S_x\cap S_y \neq \emptyset\} \mean{g(y,\beta_{n-1}+\delta_y)} \mathbf{Q}(\md y),
\end{align*}
so that
$$
T_{1,x,s} \leq 2 e^{-s} \mean{ g(x,\beta_{n-1}'+\delta_x)} n \int_{\mathbb{X}} \ind\{S_x\cap S_y \neq \emptyset\} \mean{g(y,\beta_{n-1}+\delta_y)} \mathbf{Q}(\md y).
$$
The inequality
\begin{align*}
& g(x,\beta_{n-1}+\delta_x) d_{TV}(\xi[\beta_{n-1}+\delta_x] -\delta_{f(x,\beta_{n-1}+\delta_x)},\xi_x[\beta_{n-1}]) \\
& \leq g(x,\beta_{n-1}+\delta_x) \sum_{y\in\beta_{n-1}, S_x\cap S_y \neq \emptyset} g(y,\beta_{n-1}+\delta_x),
\end{align*}
{where we used that by our assumptions $g(y,\beta_{n-1}+\delta_x)=g(y,\beta_{n-1})$ if $S_x\cap S_y=\emptyset$,} and \eqref{e:IdentityBinomial} lead to
\begin{align*}
T_{2,x,s} \leq 2 e^{-s} n \int_{\mathbb{X}} \ind\{S_x\cap S_y \neq \emptyset\} \mean{g(x,\beta_{n-2}+\delta_x+\delta_y) g(y,\beta_{n-2}+\delta_x+\delta_y)} \mathbf{Q}(\md y).
\end{align*}
Thus, we have shown that
\begin{equation}\label{e:T1+T2}
n \int_0^\infty \int_{\X} T_{1,x,s} \, \mathbf{Q}(\md x) \md s \leq E_2 \quad \text{and} \quad n \int_0^\infty \int_{\X} T_{2,x,s} \, \mathbf{Q}(\md x) \md s \leq E_3.
\end{equation}
Next, we define $N_m(x)=\beta_m(S_x)$ and $N_m'(x)=\beta_m'(S_x)$. Recall from the statement of the theorem that $\widetilde{\beta}_{x,m}$ is a binomial point process of $m$ independent points distributed according to $\mathbf{Q}_x(\cdot)=\mathbf{Q}(\cdot \cap S_x^c)/\mathbf{Q}(S_x^c)$. From \eqref{e:pslip} we obtain
\begin{align*}
T_{3,x,s} & = \big| \mean{g(x,\beta_{n-1}+\delta_x) D_{f(x,\beta_{n-1}+\delta_x)}P_sh[\xi_x(\beta_{n}']) } \\
& \qquad - \mean{g(x,\beta_{n-1}+\delta_x) D_{f(x,\beta_{n-1}+\delta_x)}P_sh(\xi_x[\beta_{n-1}]) } \big| \\
& = \big| \mean{g(x,\beta_{n-1}+\delta_x) D_{f(x,\beta_{n-1}+\delta_x)}P_sh(\xi_x[\widetilde{\beta}_{x,n-N_{n}'(x)}]) } \\
& \qquad - \mean{g(x,\beta_{n-1}+\delta_x) D_{f(x,\beta_{n-1}+\delta_x)}P_sh(\xi_x[\widetilde{\beta}_{x,n-1-N_{n-1}(x)}]) } \big| \\
& \leq 2 e^{-s} \mean{ g(x,\beta_{n-1}+\delta_x) {(\xi_x[\widetilde{\beta}_{x,n-N_{n}'(x)}] \triangle \xi_x[\widetilde{\beta}_{x,n-1-N_{n-1}(x)}])(\Y)}  }.
\end{align*} 
We denote the points of $\widetilde{\beta}_{x,m}$, by $\widetilde{X}_1,\hdots,\widetilde{X}_m$.  For $\ell\leq n$ we have
\begin{equation}\label{e:BoundXiX}
\begin{split}
& {(\xi_x[\widetilde{\beta}_{x,n-\ell}] \triangle \xi_x[\widetilde{\beta}_{x,n}])(\Y)} \\
& \leq \sum_{i=n-\ell+1}^n  \ind\{S_x\cap S_{\widetilde{X}_i} = \emptyset \} g(\widetilde{X}_i,\widetilde{\beta}_{x,n}) 
+ \sum_{i=n-\ell+1}^n \sum_{j=1}^{n-\ell} \ind\{S_x \cap S_{\widetilde{X}_j} = \emptyset \}  g(\widetilde{X}_j,\widetilde{\beta}_{x,n-\ell}) \ind\{\widetilde{X}_i \in S_{\widetilde{X}_j} \} \\
& \quad + \sum_{i=n-\ell+1}^n \sum_{j=1}^{n-\ell} \ind\{S_x \cap S_{\widetilde{X}_j} = \emptyset \}   g(\widetilde{X}_j,\widetilde{\beta}_{x,n}) \ind\{\widetilde{X}_i\in S_{\widetilde{X}_j} \} .
\end{split}
\end{equation}
Here, the first sum on the right-hand side counts the points of $\xi_x[\widetilde{\beta}_{n,x}]$ associated with points from $\widetilde{\beta}_{x,n}\setminus\widetilde{\beta}_{x,n-\ell}$, while the second and the third term bound the number of points {that are present in $\xi_x[\widetilde{\beta}_{x,n-\ell}]$ or $\xi_x[\widetilde{\beta}_{x,n}]$ and not present or different in the other point process.} {We will use the triangle inequality
$$
(\xi_x[\widetilde{\beta}_{x,n-N_{n}'(x)}] \triangle \xi_x[\widetilde{\beta}_{x,n-1-N_{n-1}(x)}])(\Y) \leq \xi_x[\widetilde{\beta}_{x,n-N_n'(x)}] \triangle \xi_x[\widetilde{\beta}_{x,n}])(\Y) + (\xi_x[\tilde{\beta}_{x,n}] \triangle \xi_x[\tilde{\beta}_{x,n-1-N_{n-1}(x)}])(\Y)
$$
in the following.} It follows from \eqref{e:BoundXiX} {with $\ell=N_n'(x)$} that
\begin{align*}
& \mean{ {(\xi_x[\widetilde{\beta}_{x,n-N_n'(x)}] \triangle \xi_x[\widetilde{\beta}_{x,n}])(\Y) } }\\
& \leq \mean{ \sum_{i=n-N_n'(x)+1}^n  \ind\{S_x \cap S_{\widetilde{X}_i} = \emptyset\} g(\widetilde{X}_i,\widetilde{\beta}_{x,n}) } \\
& \quad + \mean{ \sum_{i=n-N_n'(x)+1}^n \sum_{j=1}^{n-N_n'(x)} \ind\{S_x\cap S_{\widetilde{X}_j} = \emptyset\}  g(\widetilde{X}_j,\widetilde{\beta}_{x,n-N_n'(x)}) \ind\{\widetilde{X}_i\in S_{\widetilde{X}_j} \} } \\
& \quad + \mean{ \sum_{i=n-N_n'(x)+1}^n \sum_{j=1}^{n-N_n'(x)} \ind\{S_x\cap S_{\widetilde{X}_j} = \emptyset \}  g(\widetilde{X}_j,\widetilde{\beta}_{x,n}) \ind\{\widetilde{X}_i \in S_{\widetilde{X}_j} \} } \\
& =: U'_{1,x}+U'_{2,x}+U'_{3,x}.
\end{align*}
We define $U_{i,x}:=\mean{g(x,\beta_{n-1}+\delta_x)} U'_{i,x}$ for $i\in\{1,2,3\}$. Applying once more \eqref{e:BoundXiX} {with $\ell=N_{n-1}(x)+1$} leads to
\begin{align*}
& \mean{g(x,\beta_{n-1}+\delta_x) {(\xi_x[\tilde{\beta}_{x,n}] \triangle \xi_x[\tilde{\beta}_{x,n-1-N_{n-1}(x)}])(\Y)} }\\
& \leq \mean{ g(x,\beta_{n-1}+\delta_x) \sum_{i=n-N_{n-1}(x)}^n  \ind\{S_x\cap S_{\widetilde{X}_i}=\emptyset \} g(\widetilde{X}_i,\widetilde{\beta}_{x,n}) } \\
& \quad + \mathbb{E}\Bigg\{g(x,\beta_{n-1}+\delta_x) \\
& \qquad \qquad \times \sum_{i=n-N_{n-1}(x)}^n \sum_{j=1}^{n-1-N_{n-1}(x)} \ind\{S_x\cap S_{\widetilde{X}_j} =\emptyset\}  g(\widetilde{X}_j,\widetilde{\beta}_{x,n-1-N_{n-1}(x)}) \ind\{\widetilde{X}_i \in S_{\widetilde{X}_j} \} \Bigg\} \\
& \quad + \mean{g(x,\beta_{n-1}+\delta_x)  \sum_{i=n-N_{n-1}(x)}^n \sum_{j=1}^{n-1-N_{n-1}(x)} \ind\{S_x\cap S_{\widetilde{X}_j}=\emptyset \}  g(\widetilde{X}_j,\widetilde{\beta}_{x,n}) \ind\{\widetilde{X}_i \in S_{\widetilde{X}_j} \} } \\
& =: U_{4,x} + U_{5,x} + U_{6,x}.
\end{align*}
From \eqref{e:IdentityBinomial}, which can be adapted to the situation, where one sums only over a fraction of the points, and the independence of $N'_n(x)$ and $\widetilde{\beta}_{n,x}$ it follows that
$$
U'_{1,x} =\mean{N_n'(x)} \int_{\X} \ind\{S_x\cap S_y=\emptyset\} \mean{g(y,\widetilde{\beta}_{x,n-1}+\delta_y)} \mathbf{Q}_x(\md y),
$$
while the independence of $\beta_{n-1}$ and $\widetilde{\beta}_{n,x}$ and \eqref{e:IdentityBinomial} lead to
\begin{align*}
U_{4,x} & = \mean{g(x,\beta_{n-1}+\delta_x) (1+N_{n-1}(x))} \int_{\mathbb{X}} \ind\{S_x\cap S_y =\emptyset\} \mean{g(y,\widetilde{\beta}_{x,n-1}+\delta_y)} \mathbf{Q}_x(\md y) \\
& = \bigg( \mean{g(x,\beta_{n-1}+\delta_x)} + (n-1) \int_{S_x} \mean{g(x,\beta_{n-2}+\delta_x+\delta_z)} \mathbf{Q}(\md z)  \bigg) \\
& \quad \quad \times \int_{\X} \ind\{S_x\cap S_y=\emptyset\} \mean{g(y,\widetilde{\beta}_{x,n-1}+\delta_y)} \mathbf{Q}_x(\md y).
\end{align*}
Together with $\mean{N'_n(x)}=n \mathbf{Q}(S_x)$ and $U_{1,x}=\mean{g(x,\beta_{n-1}+\delta_x)} U'_{1,x}$ we obtain
\begin{equation}\label{e:U1+U4}
\begin{split}
U_{1,x} + U_{4,x} & = \bigg(  (1+n \mathbf{Q}(S_x) ) \mean{g(x,\beta_{n-1}+\delta_x)} + (n-1) \int_{S_x} \mean{g(x,\beta_{n-2}+\delta_x+\delta_z)} \mathbf{Q}(\md z)  \bigg) \\
& \quad \quad \times \int_{\X} \ind\{S_x\cap S_y=\emptyset\} \mean{g(y,\widetilde{\beta}_{x,n-1}+\delta_y)} \mathbf{Q}_x(\md y).
\end{split}
\end{equation}
By similar independence arguments as above and \eqref{e:IdentityBinomial}, we have
\begin{align*}
U'_{3,x} & \leq \mean{ \sum_{i=n-N_n'(x)+1}^n \sum_{j=1}^{n} \mathbf{1}\{ S_x \cap S_{\widetilde{X}_j}=\emptyset \}  g(\widetilde{X}_j,\widetilde{\beta}_{x,n}) \mathbf{1}\{\widetilde{X}_i \in S_{\widetilde{X}_j} \} } \\
& \leq n \mean{N_n'(x)} \int_{\X^2} \ind\{S_x\cap S_{y_1}=\emptyset\} \ind\{ y_2\in S_{y_1} \} \mean{g(y_1,\widetilde{\beta}_{x,n-2}+\delta_{y_1}+\delta_{y_2})} \mathbf{Q}_x^2(\md(y_1,y_2)) 
\end{align*}
and
\begin{align*}
U_{6,x} & \leq n \mean{ g(x,\beta_{n-1}+\delta_x) (1+N_{n-1}(x)) } \\
& \quad \times \int_{\X^2} \ind\{S_x\cap S_{y_1}=\emptyset\} \ind\{ y_2\in S_{y_1}\} \mean{g(y_1,\widetilde{\beta}_{x,n-2}+\delta_{y_1}+\delta_{y_2}} \mathbf{Q}^2_x(\md(y_1,y_2)) \\
& \leq n\bigg( \mean{g(x,\beta_{n-1}+\delta_x)} + n \int_{S_x} \mean{g(x,\beta_{n-2}+\delta_x+\delta_z)} \mathbf{Q}(\md z)  \bigg) \\
& \quad \times \int_{\X^2} \ind\{S_x\cap S_{y_1}=\emptyset\} \ind\{y_2\in S_{y_1}\} \mean{g(y_1,\widetilde{\beta}_{x,n-2}+\delta_{y_1}+\delta_{y_2}} \mathbf{Q}^2_x(\md(y_1,y_2)).
\end{align*}
From the definition of $U_{3,x}$ and $\mean{N'_n(x)}=n \mathbf{Q}(S_x)$ we derive
\begin{equation}\label{e:U3+U6}
\begin{split}
U_{3,x}+U_{6,x} & \leq n\bigg( (1+n \mathbf{Q}(S_x)) \mean{g(x,\beta_{n-1}+\delta_x)} + n \int_{S_x} \mean{g(x,\beta_{n-2}+\delta_x+\delta_z)} \mathbf{Q}(\md z)  \bigg) \\
& \quad \times \int_{\X^2} \ind\{S_x\cap S_{y_1}=\emptyset\} \ind\{y_2\in S_{y_1}\} \mean{g(y_1,\widetilde{\beta}_{x,n-2}+\delta_{y_1}+\delta_{y_2}} \mathbf{Q}^2_x(\md(y_1,y_2)).
\end{split}
\end{equation}
Combining \eqref{e:U1+U4} and \eqref{e:U3+U6} yields
\begin{equation}\label{e:E4}
2n \int_{\X} \left( U_{1,x}+U_{3,x}+U_{4,x}+U_{6,x} \right) \, \mathbf{Q}(\md x) \leq E_4.
\end{equation}
By \eqref{e:AssumptionRi} and the assumption $\cS(z,\omega)\subset S_z$ for all $z\in \X$ and $\omega\in\bN_{\X}$, we obtain
\begin{equation}\label{e:Identity_g}
\ind\{z_1\notin S_{z_2}\} g(z_2,\omega+\delta_{z_1}) = \ind\{z_1\notin S_{z_2}\} g(z_2,\omega)
\end{equation}
for $z_1,z_2\in\X$ and $\omega\in\bN_{\X}$ {with $z_2\in\omega$.} For $x\in\X$ we deduce from \eqref{e:IdentityBinomial} and \eqref{e:Identity_g} that 
\begin{align*}
U'_{2,x} & = \mean{ N_n'(x) \sum_{j=1}^{n-N_n'(x)} \ind\{S_x\cap S_{\widetilde{X}_j} = \emptyset \}  g(\widetilde{X}_j,\widetilde{\beta}_{x,n-N_n'(x)}) \mathbf{Q}_x(S_{\widetilde{X}_j}) } \\
& = \mean{ \sum_{(x_1,x_2)\in\beta_{n,\neq}^2} \ind\{x_1\in S_x\} \ind\{S_x\cap S_{x_2}=\emptyset\} g(x_2,\beta_n) \mathbf{Q}_x(S_{x_2}) } \\
& = n(n-1) \mathbf{Q}(S_x) \int_{\X} \ind\{S_x\cap S_y=\emptyset\} \mean{g(y,\beta_{n-2}+\delta_y)} \mathbf{Q}_x(S_y) \mathbf{Q}(\md y)
\end{align*}
so that, with $U_{2,x}=\mean{g(x,\beta_{n-1}+\delta_x)}U'_{2,x}$,
\begin{equation}\label{e:E5}
2 n\int_{\X} U_{2,x} \, \mathbf{Q}(\md x) \leq E_5.
\end{equation}
Finally, applying again \eqref{e:IdentityBinomial} and \eqref{e:Identity_g} leads to 
\begin{align*}
U_{5,x} & = \mathbb{E} \bigg\{ g(x,\beta_{n-1}+\delta_x) (1+N_{n-1}(x)) \sum_{j=1}^{n-1-N_{n-1}(x)} \ind\{S_x\cap S_{\widetilde{X}_j}=\emptyset\}  g(\widetilde{X}_j,\widetilde{\beta}_{x,n-1-N_{n-1}(x)}) \mathbf{Q}_x(S_{\widetilde{X}_j}) \bigg\} \\
& = \mean{ g(x,\beta_{n-1}+\delta_x) \sum_{y\in \beta_{n-1}} \ind\{S_x\cap S_y = \emptyset\} g(y,\beta_{n-1}) \mathbf{Q}_x(S_y) } \\
& \quad + \mean{ g(x,\beta_{n-1}+\delta_x) \sum_{(y_1,y_2)\in\beta_{n-1,\neq}^2}  \ind\{y_1\in S_x\} \ind\{S_x\cap S_{y_2}=\emptyset\} g(y_2,\beta_{n-1}) \mathbf{Q}_x(S_{y_2}) } \\
& \leq n \int_{\X} \ind\{S_x\cap S_y = \emptyset\} \mean{ g(x,\beta_{n-2}+\delta_x) g(y,\beta_{n-2}+\delta_y) }  \mathbf{Q}_x(S_y) \mathbf{Q}(\md y) \\
& \quad { + n^2 \int_{\X} \int_{S_x} \ind\{S_x\cap S_y = \emptyset\} \mean{ g(x,\beta_{n-3}+\delta_x+\delta_z) g(y,\beta_{n-3}+\delta_y) } \mathbf{Q}_x(S_y) \mathbf{Q}(\md z)  \mathbf{Q}(\md y) } 
\end{align*}
for $x\in\X$, which implies
\begin{equation}\label{e:E6}
2n \int_{\X} U_{5,x} \, \mathbf{Q}(\md x) \leq E_6.
\end{equation}
Combining \eqref{e:E4}, \eqref{e:E5} and \eqref{e:E6} provides a bound for the integral over $T_{3,x,s}$, which together with \eqref{e:InequalityTs} and \eqref{e:T1+T2} completes the proof.
\end{proof}

\section{Applications}\label{sec:Applications}

In this section we present applications of our results to (a) the  critical points of a random distance function and (b) {the volume (with respect to the intensity measure) of large $k$-nearest neighbor balls.}  Though some Poisson approximation results are known in both of these models, our results extend these in two ways. Firstly, we consider more general point processes than those considered before and secondly, we provide rates of convergence under a stronger metric. More details on comparison with the existing literature for these specific applications will be given in the respective subsections. 

We wish to point out that these applications are to illustrate our generic results and it is conceivable that many more such applications of our results would be possible.  For example, one may consider extremes of circum-radii and in-radii of Poisson-Voronoi tessellations as in \cite{Calkaextreme2014,Chenaviertessellations2014,PianoforteSchulte2021} and also other stabilizing statistics as in \cite{Otto2020}. For example, we can deduce immediately from Theorem \ref{t:qldep} and the proofs in \cite{Otto2020} that \cite[Theorems 4.2, 5.5 and 6.5]{Otto2020} hold under the stronger KR distance. Using Theorems \ref{t:KR_coupling} or \ref{t:qldep}, we may extend many Poisson approximation results to Poisson \emph{process} approximation results. For example, a coupling similar to that in Theorem \ref{t:KR_coupling} was used in \cite[Theorem 2.8]{Iyer2020} to prove Poisson convergence for the number of isolated faces in a Vietoris-Rips complex.  We expect that other Poisson approximation results in \cite{penrose_inhomogeneous_2018} could also be extended using our framework (see Remark \ref{rem:KRcoupling}).  {Also, we  believe that Poisson convergence results (for example,  see \cite[Theorem 2.1]{Owada2017}) based on dependency graph method \cite{Arratia1989} could be extended to a Poisson process approximation using our Theorem \ref{t:qldep}.} Another point to emphasize is that most of the afore-mentioned articles consider only approximation for point processes induced by Poisson point processes but our bounds will help extending them to point processes induced by binomial point processes as well.  For example,  see Theorem \ref{t:knnapproxbin}.   

\subsection{Critical points for the random distance function}
\label{s:critpoint_dist}

Let $\omega$ be a point configuration in some metric space.
In this section we are interested in the space generated by the union of balls around $\omega$,
\[
	B_r(\omega) := \bigcup_{p\in \omega} B_r(p),
\]
{where $B_r(p)$ denotes the closed ball of radius $r$ centred at $p$.}  In \cite{kahle_random_2011} the theoretical study of \emph{homology} of $B_r(\cdot)$ taken over random point processes was initiated.  Briefly, homology is an algebraic-topological structure representing information about cycles in various dimensions, where $0$-dimensional cycles correspond to connected components, $1$-dimensional cycles correspond to loops surrounding ``holes", $2$-dimensional cycles correspond to surfaces enclosing  ``air pockets", etc. (cf. \cite{hatcher_algebraic_2002}). Increasing the radius $r$, \emph{Morse Theory} (cf. \cite{gershkovich_morse_1997,milnor_morse_1963}) states that changes in the homology of $B_r(\omega)$ occur at critical levels of the distance function
\[	
\rho(x;\omega) := \min_{p\in \omega} \dist(x,p),
\]
where $\dist(\cdot,\cdot)$ is the distance under the corresponding metric. For every critical point, Morse theory also assigns an \emph{index} which, roughly speaking, counts the number of independent directions around the critical point, along which the distance function is decreasing. Critical points of index $k$ can then affect the homology of $B_r(\eta)$ either in  dimension $k$ (creating new $k$-dimensional cycles) or in dimension $k-1$ (terminating a $(k-1)$-dimensional cycle). Thus, the work in \cite{bobrowski_distance_2014,Bobrowski2019} focused on the critical points and their indexes, as a proxy to the homology. We shall define critical points more formally below. Note that  we will not elaborate more on the connection between critical points and homology as this will involve introducing algebraic topology basics. For more on the homology of $B_r(\omega)$, we refer the reader to the survey \cite{Bobrowski2018}.

{The work in \cite{Bobrowski2019} focused on a homogeneous Poisson process defined on a flat torus $\X = \T^d = \R^d/\Z^d$ (which can be thought of as 
the unit box $[0,1]^d$ with a periodic boundary). Specifically, it considered a Poisson process $\eta=\eta_n$ on $\T^d$ with a fixed $n$.} The main objective there was to characterize the phase transition for \emph{homological connectivity}, i.e., where the $k$-th homology of the random union $B_r(\eta_n)$ converges to the $k$-th homology of the underlying torus $\T^d$. The main part of the proof there, shows that the very last obstructions to the $k$-th homological connectivity are in one-to-one correspondence with critical points of index $k$. In fact, it is shown in \cite[Section 8]{Bobrowski2019} that in terms of \emph{persistent homology} these obstructions form infinitesimally small persistence intervals that are the very last ones to appear in the persistence barcode. 
The conditions for points to be critical, as presented in \cite[Section 2.3]{Bobrowski2019} are localized in the sense of Section \ref{s:locdep}. Hence, we are interested in providing a {limit} theorem for this process of obstructions.

Let $r_n$ be defined via
\eqb\label{e:r_n}
	a_n := \omega_d nr_n^d = \log n + (k-1)\logg n + \alpha_0,
\eqe
where $\omega_d$ is the volume of a {$d$-dimensional} unit ball, and $\alpha_0\in \R$ is fixed. 
We call the values of $r_n$ satisfying \eqref{e:r_n} the `$k$-homological connectivity' regime. Let $R_n$ be any value satisfying 
\eqb\label{e:R_n}
\limninf R_n =0,\quad\text{and}\quad\limninf r_n/R_n = 0.
\eqe
The analysis in \cite{Bobrowski2019} focused on the number of critical points $c\in \T^d$ with Morse index $k$, such that $\rho(c;\eta_n) \in (r_n,R_n]$.
Defining by $C_{k,n}$ the number of such critical points, in \cite[Proposition 4.1]{Bobrowski2019} the following was proved.

\begin{prop}\label{prop:crit_moments}
Let $\alpha_0 \in \R$, and let $r_n,R_n$ satisfy \eqref{e:r_n} {and} \eqref{e:R_n}, respectively. Then
\[
	\limninf \mean{C_{k,n}} = \limninf \var{C_{k,n}} = D_k e^{-\alpha_0},
\]
where $D_k$ is a known constant.
\end{prop}
Note that the limiting results are independent of the choice of $R_n$ (provided that \eqref{e:R_n} is satisfied).
To prove Proposition \ref{prop:crit_moments}, \cite{Bobrowski2019} used the fact that critical points of index $k$ are generated by subsets  $\bx\in(\eta_n)^{k+1}_{\ne}$. Defining
\eqb\label{e:g_rR}
g(\bx, \eta_n) := \ind\set{\text{$\bx\subset \eta_n$ generates a critical point of index $k$, with $\rho\in (r_n,R_n]$}},
\eqe
then we can write
\[
	C_{k,n} := \frac{1}{(k+1)!} \sum_{\bx\in(\eta_n)^{k+1}_{\ne}} g(\bx,\eta_n).
\]	
Our goal here is not only to prove a Poisson limit for the random variable $C_{k,n}$ but to provide an elaborate point-process limit for the actual configuration of critical points that appear in $(r_n,R_n]$.
To this end, we require a few more definitions. 

We will consider subsets $\bx\in (\T^d)^{k+1}$ that are (a) in general position, and (b) contained in a ball of radius $R_n$. For such subsets we can define $c(\bx)$ and $\rho(\bx)$ as the center and radius of the unique $(k-1)$-sphere containing $\bx$. In \cite[Lemma 2.3]{Bobrowski2019} it was shown that every critical point of $\rho(\cdot\ ;\eta_n)$ with Morse index $k$, is of the form $c(\bx)$ for some $\bx \in (\eta_n)^{k+1}_{\ne}$, and  the corresponding critical value satisfies $\rho(c(\bx);\eta_n) = \rho(\bx)$.
Next, following \eqref{e:r_n} we define
\[
	\alpha(\bx) := \omega_d n\rho(\bx)^d - \log n - (k-1)\logg n.
\]
Our goal is to define a point process on $\T^d\times\R$ representing pairs of the form $(c(\bx), \alpha(\bx))$ for those critical points in $(r_n,R_n]$, by defining
\[
	\xi_k = \xi_k[\eta_n] := \frac{1}{(k+1)!} \sum_{\bx\in(\eta_n)^{k+1}_{\ne}} g(\bx,\eta_n) \delta_{f(\bx)},
\]	
where $f(\bx) = (c(\bx),\alpha(\bx))$, and $r_n$ satisfies \eqref{e:r_n}.
In other words, the domain space here is $\X = \T^d$ and the image space is  $\Y: = \T^d \times \R_0$, where  $\R_0 := [\alpha_0,\infty)$. From Proposition \ref{prop:crit_moments} we know that 
\eqb\label{e:lim_moments}
	\limninf \mean{\xi_k(\Y)} = \limninf \var{\xi_k(\Y)} = D_k e^{-\alpha_0}.
\eqe
We will prove that $\xi_k$ converges to a Poisson process on $\Y$.

\begin{thm}\label{thm:crit_pts}
Let $\alpha_0 \in \R$, and $\xi_k$ as defined above. Then for $n\ge 3$ and $R_n = \sqrt{r_n}$, we have
\[
	\dkr(\xi_k, \zeta_k) \le C_{\alpha_0} (\logg n)^d(\log n)^{-\frac{d-k}{d+1}},
\]
for some $C_{\alpha_0}>0$, and where $\zeta_k$ is a Poisson process on $\Y = \T^d\times\R_0$, with intensity   $D_k e^{-\alpha}\md\alpha \md c$.
In particular, this implies that $\xi_k \xrightarrow{\bf KR} \zeta_k$ as $n\to\infty$.
\end{thm}
In other words, the process of all pairs $(c,\alpha)$ representing critical points and critical values, has a limit of a Poisson process which is homogeneous in $c$ and has an exponentially decaying intensity in $\alpha$.

\begin{rem}
\label{r:critpts}
{\rm
\begin{enumerate}
\item Theorem 8.1 in \cite{Bobrowski2019} presents a weaker statement than Theorem \ref{thm:crit_pts} here. The limiting point process in \cite{Bobrowski2019} was only for the critical radii, while here we show that a Poisson limit also extends to the combined location+radius  point process. Regarding the proofs, the calculations in \cite{Bobrowski2019} provide all the moment estimates needed to invoke Theorem \ref{t:qldep}. However, the key ingredient needed to prove the Poisson limit in \cite{Bobrowski2019} is in fact Theorem  \ref{t:qldep} which appears here for the first time.

\item We can also conclude from Theorem  \ref{thm:crit_pts} the convergence of the entire point process of critical points and distances (i.e., on $\T^d\times\R$). More precisely, consider
$$ \xi'_k[\eta_n] := \frac{1}{(k+1)!}\sum_{\bx \in (\eta_n)^{k+1}_{\neq}} \ind\{\mbox{$\bx \subset \eta_n$ generates a critical point of index $k$ with $\rho(\bx) \leq R_n$}\}\delta_{f(\bx)},$$
where $f(\bx), \rho(\bx)$ are as defined above. Since Theorem \ref{thm:crit_pts} holds for all $\alpha_0 \in \R$, {by the characterization of convergence in distribution of point processes in \cite[Theorem 16.16]{KallenbergFoundations}}, we have that $\xi'_k[\eta_n] \stackrel{d}{\to} \zeta'_k$, where $\zeta'_k$ is a Poisson process on $\Y = \T^d\times \R$, with intensity $D_k e^{-\alpha}\md\alpha \md c$.

\item The choice of $R_n = \sqrt{r_n}$ was required in order to get the bound presented in Theorem \ref{thm:crit_pts}. In terms of the homological connectivity phenomenon, this choice has no practical consequence. In fact it can be shown that even at radius $r=2r_n$ with high probability $B_r(\eta_n)$ covers the torus $\T^d$ completely. This implies that homology exhibits no further changes, and the only changes we expect to see are at radii smaller than $2r_n \ll \sqrt{r_n}$.
\end{enumerate}
}
\end{rem}

\begin{proof}[Proof of Theorem \ref{thm:crit_pts}]
Following the notation above, we denote by $\bL$ and $\bM$ the intensity measures of $\xi_k$ and $\zeta_k$, respectively. Denote $\X := \T^d$, and note that since $\eta_n$ is a homogeneous Poisson process with rate $n$ on $\X$, we have that $\bK(\md x) = n\md x$. Also, setting $S_{\bx} = B_{\rho(\bx)}(c(\bx))$ then the definitions for the critical points in \cite{Bobrowski2019} imply that $g(\bx,\eta_n)$ is indeed localized to $S_{\bx}$, i.e., condition \eqref{e:AssumptionRi} holds with $\cS(\bx,\omega) \equiv S_{\bx}$.

Using Theorem \ref{t:qldep} we have,
\eqb\label{e:crit_pt_KR}
\begin{split}
&\dkr(\xi_k,\zeta_k) \le  d_{TV}({\bL},{\bM}) + 2\{\var{\xi_k(\Y)} - \mean{\xi_k(\Y)} \}\\
&\quad + \param{\frac{2}{(k+1)!}}^2 \int_{\X^{k+1}} \int_{\X^{k+1}}\ind\set{S_{\bz}\cap S_{\bx}\ne \emptyset} \mean{g(\bx,{\eta_n}+\delta_{\bx})} \mean{g(\bz,{\eta_n}+\delta_{\bz})} \bK^{k+1}(\md\bz) \bK^{k+1}(\md\bx).
\end{split}
\eqe

First, recall that $\zeta_k$ is a Poisson process on $\Y = \T^d\times\R_0$, with intensity $D_k e^{-\alpha}\md\alpha \md c$. Then for any measurable $A\subset \Y$, 
\[
	\bM(A) = \mean{\zeta_k(A)}  = D_k\int_A  e^{-\alpha}\md\alpha \md c.
\]
Next, the calculations in \cite{Bobrowski2019} (proof of Lemma 4.3) show that 
\[
	\bL(A) = \mean{\xi_k(A)}  = D_k n \int_{\hat A} s^{k-1}e^{-s} \ind\set{s\le \omega_d n R_n^d } \md s \md c,
\]
where 
\[
	\hat A := \set{(c,s) {:} (c,\alpha)\in A,\ s = \log n + (k-1)\logg n + \alpha}.
\]	
Therefore, defining $I_n(\alpha):= \ind\set{\log n + (k-1)\logg n + \alpha \le \omega_d n R_n^d}$, we have
\begin{align}
	\bL(A)  &= D_k n \int_{A} (\log n + (k-1)\logg n + \alpha)^{k-1}e^{-(\log n + (k-1)\logg n + \alpha)} I_n(\alpha) \md\alpha \md c \nonumber \\
\label{e:defnLcrit} &= D_k \int_{A} \param{1+\frac{ (k-1)\logg n + \alpha}{\log n}}^{k-1} e^{- \alpha} I_n(\alpha) \md\alpha \md c.
\end{align}
If $n$ is large enough, then $\alpha_0 > -\logg n$, and we have
\eqb\label{e:cp_bound_1}
\splitb
	|\bL(A)-\bM(A)| &= D_k \abs{\int_{A} \param{I_n(\alpha)\param{1+\frac{(k-1)\logg n+ \alpha}{\log n}}^{k-1}-1}e^{-\alpha} \md\alpha \md c}\\
&\le D_k \sum_{i=1}^{k-1}\binom{k-1}{i} \int_{A}\param{\frac{(k-1)\logg n + \alpha}{\log n}}^i e^{-\alpha} \md\alpha \md c + D_k \int_{\R} (1-I_n(\alpha)) e^{-\alpha} \md \alpha\\ 
&= D_k \sum_{i=1}^{k-1}\sum_{j=0}^{i}\binom{k-1}{i}\binom{i}{j} \frac{((k-1)\logg n)^{i-j}}{(\log n)^i}\int_{A} \alpha^j e^{-\alpha}\md\alpha \md c \\
& \quad + D_k e^{-\omega_d n R_n^d +\log n +(k-1)\logg n}\\
&= O\param{\frac{\logg n}{\log n}},
\splite
\eqe
where we used the fact that $\int_{\R_0}\alpha^j e^{-\alpha}\md\alpha < \infty$ and that $R_n=\sqrt{r_n}$.

The fact that $\var{\xi_k(\Y)} - \mean{\xi_k(\Y)} \to 0$ is given by \eqref{e:lim_moments}. 
To get an upper bound, we follow more carefully the proof of Proposition 4.1 in \cite{Bobrowski2019}. In \cite[Equation (8.6)]{Bobrowski2019} it is shown that
\[
\var{\xi_k(\Y)} - \mean{\xi_k(\Y)}  = \sum_{j=1}^kI_j + (I_0-\mean{\xi_k(\Y)}^2),
\]
where the terms $I_j$ are defined there. In \cite[Equation (8.14)]{Bobrowski2019} it is shown that for $1\le j \le k$, and for small enough $\delta_j,\eps_j$ we have
\[
I_j \le C_1 n a_n^{k-1}e^{-a_n}\param{\eps_j^{k+2-j}a_n^{k+1-j} +(\delta_j/\eps_j)^{d-k}(\delta_j a_n )^{k+1-j} + a_n^{k+1-j}e^{-C_2\delta_j a_n}},
\]
and in \cite[Equation (8.15)]{Bobrowski2019} it is shown that
\[
I_0-\mean{\xi_k(\Y)}^2 \le C_3 n a_n^{k-1}e^{-a_n}\param{\eps_0^{d+2}a_n^{k+1}+a_n^{k+1}e^{-C_4\eps_0 a_n}},
\]
where $C_1,\ldots,C_4$ are some positive constants.
The choice of $\eps_j = a_n^{-\param{1-\frac{1}{d+2-j}}}$ and $\delta_j = \frac{k+2-j}{C_2} \frac{\log a_n}{a_n}$ yields that for $1 \leq j \leq k$
\begin{align*}
I_j &  \leq C_1 n a_n^{k-1}e^{-a_n} \param{ a_n^{-\frac{d-k}{d+2-j}} +\param{\frac{k+2-j}{C_2}\log a_n}^{d-j+1}a_n^{-\frac{d-k}{d+2-j}} + a_n^{-1} } \\
& = O\param{(\log a_n)^{d-j+1}a_n^{-\frac{d-k}{d+2-j}}} = O\param{(\logg n)^{d}(\log n)^{-\frac{d-k}{d+1}}}.
\end{align*}
In addition, taking $\eps_0 = \frac{k+2}{C_4}\frac{\log a_n}{a_n}$ yields
\[
I_0-\mean{\xi_k(\Y)}^2 = O\param{a_n^{-1}} = o\param{(\logg n)^d (\log n)^{-\frac{d-k}{d+1}}}.
\]
Therefore, we have
\eqb\label{e:cp_bound_2}
\var{\xi_k(\Y)} - \mean{\xi_k(\Y)} = O\param{(\logg n)^d(\log n)^{-\frac{d-k}{d+1}}}.
\eqe
We are left now with bounding the last term in \eqref{e:crit_pt_KR}.

Fix $\bx$ and note that if  $S_{\bz} \cap S_{\bx} \ne \emptyset$, then $|c(\bz)-c(\bx)| \le \rho(\bz)+\rho(\bx) \le 2R_n$. Therefore, using  \eqref{e:defnLcrit}, we derive that
\[
\splitb
& \frac{1}{(k+1)!}\int_{\X^{k+1}}\ind\set{S_{\bz}\cap S_{\bx}\ne \emptyset} \mean{g(\bz,\eta_n+\delta_{\bz})} \bK^{k+1}(\md\bz)\\
&\quad \le \frac{1}{(k+1)!}\int_{\X^{k+1}} 
\mean{\ind\set{c(\bz)\in B_{2R_n}(c(\bx))}g(\bz,\eta_n+\delta_{\bz})} \bK^{k+1}(\md\bz)\\
&\quad= \bL(B_{2R_n}(c(\bx))\times \R_0) = \vol(B_{2R_n}(c(\bx)) \cdot \bL(\T^d\times\R_0) = \omega_d (2R_n)^d D_k e^{-\alpha_0} ,
\splite
\]
where we used the fact that $\vol(\T^d) = 1$. Finally, we can conclude that
\begin{align}
&\frac{1}{((k+1)!)^2}\int_{\X^{k+1}} \int_{\X^{k+1}}\ind\set{S_{\bz}\cap S_{\bx}\ne \emptyset} \mean{g(\bx,\eta_n+\delta_{\bx})} \mean{g(\bz,\eta_n+\delta_{\bz})} \bK^{k+1}(\md\bz) \bK^{k+1}(\md\bx) \nonumber \\
\label{e:cp_bound_3}  &\quad\le \omega_d (2R_n)^d D_k e^{-\alpha_0} \frac{1}{(k+1)!}\int_{\X^{k+1}}  \mean{g(\bz,\eta_n+\delta_{\bz})}\bK^{k+1}(\md\bx) = O(R_n^d).
\end{align}
Since $R_n = \sqrt{r_n}$, we have $R_n^d = O(\sqrt{\log n/n})$.
Combining  the bounds in   \eqref{e:cp_bound_1}-\eqref{e:cp_bound_3} completes the proof.
\end{proof}

\subsection{Large $k$-nearest neighbor balls}
\label{s:knn}
{We look at the point process of the scaled volumes of $k$-nearest neighbor balls. } We shall consider the setup as in \cite[Section 4]{Otto2020}. This is a well-studied statistic in computational and stochastic geometry with varied applications (see \cite[Pg. 342]{Penrose1997}). It is also closely related to another important object, the minimal spanning tree. We remark on this briefly in our discussion after Theorem \ref{t:knnapprox}. 

Let $\X = \T^d, d \geq 2$ as in the previous section and identify it with $[0,1]^d$, and let $\bK$ be a finite measure on $(\X,\cX)$. Define for $k\ge 1$ and $\omega \in \widehat{\bN}_{\X}$,
\[
\splitb
R_k(x,\omega) &:= \inf \{r > 0 : \omega(B_r(x) \setminus \{x\} ) {\geq} k \},\\
\bK_k(x,\omega) &:= \bK(B_{R_k(x,\omega)}(x)),
\splite
\]
where $B_r(x)$ is the {closed} ball of radius $r$ around $x$. In other words, $R_k$ is the $k$-nearest neighbor distance of $x$, and $\bK_k$ is the measure of the ball of radius $R_k$ around $x$. Let $\Y = \X\times \R$, and define
$$ \xi[\omega] := \sum_{x \in \omega} \delta_{(x, n\bK_k(x,\omega) - a_n)},$$
where
\begin{equation}
\label{e:a_n}
a_n = \log n + (k-1)\log \log n - \log (k-1)!.
\end{equation}
In other words, {for an underlying point process $\nu$, $\xi[\nu]$} is a point process of pairs, each pair consisting of a point and the scaled volume of its $k$-nearest neighbor ball. The scaling is chosen such that we count only the extremal (maximal) ones. 
\begin{thm}
\label{t:knnapprox}
Let $d \geq 2$ and $k \geq 1$. Let $\eta_n$ be a Poisson process with intensity measure $n\bK$ and assume that $\bK$ is a probability measure with a density $\lambda : \X \to (0,\infty)$ such that $0 < \lambda_- \le \lambda(x) \le \lambda_+ < \infty$ for all $x\in \X$. Let $\zeta$ be a Poisson process with intensity measure $\bM(\md x,\md t) = \lambda(x) \, \md x \, e^{-t} \, \md t$ on $\X \times \R$.  Then, for any $b_0 \in \R$, there exists a constant $C\in(0,\infty)$ only depending on $b_0$, $k$, $d$, $\lambda_+$ and $\lambda_-$ such that
\begin{equation}\label{e:knnbd}
\dkr(\xi[\eta_n] \cap (\X \times (b_0,\infty)),\zeta \cap (\X \times (b_0,\infty))) \leq C \frac{\log \log n}{\log n}
\end{equation}
for all {$n\geq 3$.} Moreover, $\xi[\eta_n] \stackrel{d}{\to} \zeta$ as $n \to \infty$.
\end{thm}
The constant $C$ in the theorem can be deduced from the bounds in \eqref{e:knngt1}, \eqref{e:knngt2}, \eqref{e:knngt3}, \eqref{e:kNNE31} and \eqref{e:kNNE32}. The convergence of $\xi[\eta_n]$ was obtained in \cite[Theorem 2]{Penrose1997} in the case of homogeneous Poisson processes.  Our above theorem considers inhomogeneous Poisson processes as well. Further \cite[Theorem 2]{Penrose1997} is used to prove Poisson process convergence of extremal edge lengths of the minimal spanning tree built on the complete graph on $\eta_n$ with Euclidean distances as the edge-weights (see \cite[Theorem 3]{Penrose1997}). The proof technique of \cite[Theorem 2]{Penrose1997} is to discretize and use the Chen-Stein bound of \cite{Arratia1989}.

Theorem 4.2 in \cite{Otto2020} gives bounds under total variation distance for the projection of $\xi[\eta_n]  \cap (\X \times (0,\infty))$ to $\X$. We will now show that using our Theorem \ref{t:qldep} with some additional compuations compared to  \cite[Theorem 4.2]{Otto2020},  we obtain improved rates of convergence for the more general point process $\xi[\eta_n]  \cap (\X \times (b_0,\infty))$ for any $b_0 \in \R$.
\begin{proof}
Let $b_0 \in \R$ be fixed. Set $g(x,\omega) := \ind\set{n\bK_k(x,\omega) - a_n > b_0}$ and
$f(x,\omega) := (x, n\bK_k(x,\omega) - a_n)$.  Define
$$ \xi_0[\omega] := \sum_{x \in \omega} g(x,\omega)\delta_{f(x,\omega)}, \, \, \omega \in \widehat{\bN}_{\X},$$
and observe that $\xi_0[\omega] = \xi[\omega] \cap (\X \times (b_0,\infty)).$  We put
$$r_n(x,u) := \inf \{r : n\bK(B_r(x)) \geq a_n + u \}.$$
We will apply Theorem \ref{t:qldep} with $f$ and $g$ as above to compare {$\xi_0[\eta_n]$ with the restriction of $\zeta$ to $\X\times (b_0,\infty)$.} Denoting $\cS(x,\omega) = B_{R_k(x,\omega)}(x)$, observe that {$f$ and $g$ are both localized to $\cS$,} which is a stopping set. Further, we set $S_x = B_{r_n(x,b)}(x)$ for a fixed {$b\in(\max\{0,b_0\},\infty)$,} which will be chosen later as a function of $n$.

We shall compute one-by-one the terms in the KR distance bound provided by Theorem \ref{t:qldep}. Observe that since  we are summing over points, the last term $E_4$ is absent. Throughout this proof all inequalities are supposed to hold for all $n\in\N$ such that {$a_n+b_0>0$.}

Note that for $x\in\X$, we have  $\{n\bK_k(x,\eta_n+\delta_x) - a_n >u \} = \{ \eta_n(B_{r_n(x,u)}(x)) < k \}$ ({a.s.}). Consider the first term in the bound of Theorem \ref{t:qldep}. Define $\varrho_n:=n\lambda_- \omega_d/2^d - a_n$, where $\omega_d$ is the volume of the unit ball in $\R^d$. Denote the intensity measure of $\xi_0$ by $\bL$. Then for $B \in \cX$ and $u \in (b_0,\varrho_n)$ we have by the Mecke formula that 
\begin{align*}
\bL(B \times (u,\infty)) &= n \int_{B} \prob{n\bK_k(x,\eta_n + \delta_x) - a_n > u} \lambda(x) \md x  = n  \int_{B} \prob{\eta_n(B_{r_n(x,u)}(x)) \leq k-1} \lambda(x) \md x \\
& = n \bK(B) \sum_{i=0}^{k-1}e^{-(a_n+u)}\frac{(a_n+u)^i}{i!}  =  e^{-u} \bK(B) \sum_{i=0}^{k-1} \frac{(k-1)!(a_n+u)^i}{ i! (\log n)^{k-1}},
\end{align*}
where we used \eqref{e:a_n} in the last equality. Thus, the density $\ell$ of $\bL$ satisfies
$$
\ell(x,u) = \lambda(x) e^{-u} \frac{(a_n+u)^{k-1}}{(\log n)^{k-1}}
$$
for {$x\in\X$ and} $u\in(b_0,\varrho_n)$. Let $\bM_0$ be the intensity measure of $\zeta_0 := \zeta \cap (\X \times (b_0,\infty))$. Now we obtain
\begin{equation}
\label{e:knngt1}
d_{TV}(\bL,\bM_0)\leq \int_{{b_0}}^{\infty} \big| e^{-u} \frac{(a_n+u)^{k-1}}{(\log n)^{k-1}} -e^{-u}  \big| \, \md u + \bL(\X\times (\varrho_n,\infty) ) +\bM_0(\X\times (\varrho_n,\infty) ) \leq  C_0\frac{\log \log n}{\log n}
\end{equation}
with a constant $C_0\in(0,\infty)$ only depending on $b_0$, $k$, $d$ and $\lambda_-$.
Now, we compute the integral in the second term of the bound in Theorem \ref{t:qldep}.  Similar to the above calculation,  we can derive that
\begin{align}
\label{e:knngt2} 
E_1 & \leq {2} n  \int_{\X} \prob{\eta_n(B_{r_n(x,b)}(x)) \leq k-1} \lambda(x) \md x  \leq C_1 e^{-b}\param{1 + \frac{(k-1) \log \log n +b}{\log n}}^{k-1}
\end{align}
for $b<\varrho_n$ with $C_1:=2 k!$.

Observing that $\{\cS(x,\omega) \subset S_x \} = \{R_k(x,\omega) \leq r_n(x,b) \}$ and using the definition of $g$, the integral in the third term of the bound in Theorem \ref{t:qldep} can be bounded from above by
\begin{align}
E_2 &\leq {2} n^2 \int_{\X}\int_{\X} \prob{R_k(x,\eta_n + \delta_x) \in (r_n(x,b_0),r_n(x,b)]}\prob{R_k(z,\eta_n + \delta_z) \in (r_n(z,b_0),r_n(z,b)]} \no \\ 
& \qquad \qquad \qquad \times \ind\set{|x-z| \leq r_n(x,b) + r_n(z,b)} \lambda(x)\lambda(z) \, \md x \, \md z \no  \\
& \leq {2} n^2 \int_{\X}\int_{\X} \prob{\eta_n(B_{r_n(x,b_0)}(x)) \leq k-1}\prob{\eta_n(B_{r_n(z,b_0)}(z)) \leq k-1} \no \\
& \qquad \qquad \qquad \times \ind\set{|x-z| \leq r_n(x,b) + r_n(z,b)} \lambda(x)\lambda(z) \, \md x \, \md z, \no \\
& \leq {2} \param{e^{-b_0} \sum_{i=0}^{k-1} \frac{(k-1)!(a_n+b_0)^i}{ i! (\log n)^{k-1}}}^2  \int_{\X}\int_{\X} \ind\set{|x-z| \leq r_n(x,b) + r_n(z,b)} \lambda(x)\lambda(z) \, \md x \, \md z \no \\
\label{e:knngt3} & \leq C_2 \frac{a_n + b}{n},
\end{align}
where  in the last bound we have used that $\sup_{x \in \X}\omega_dnr_n(x,b)^d \leq \lambda_-^{-1} \sup_{x \in \X} n\bK(B_{r_n(x,b)}) =\frac{a_n + b}{\lambda_-}$. Note that the constant $C_2\in(0,\infty)$ depends on $b_0$, $k$, $\lambda_+$ and $\lambda_-$. 

Similarly to the above bound, the fourth term in Theorem \ref{t:qldep} can be bounded by
\begin{align*}
E_3 &\leq {2} n^2 \int_{\X}\int_{\X} \prob{[\eta_n + \delta_z](B_{r_n(x,b_0)}(x)) \leq k-1, [\eta_n + \delta_x](B_{r_n(z,b_0)}(z)) \leq k-1} \\
 & \qquad \qquad \qquad \times \ind\set{|x-z| \leq r_n(x,b) + r_n(z,b)} \lambda(x)\lambda(z) \, \md x \, \md z \\
& \leq {2} n^2 \lambda_+^2 \int_{\X}\int_{\X} \prob{[\eta_n + \delta_z](B_{r_n(x,b_0)}(x)) \leq k-1, [\eta_n + \delta_x](B_{r_n(z,b_0)}(z)) \leq k-1} \\
 & \qquad \qquad \qquad \times \ind\set{|x-z|\leq \min\{r_n(x,b_0),r_n(z,b_0)\}} \ind\set{|x-z| \leq r_n(x,b) + r_n(z,b)} \, \md x \, \md z \\
& \quad + {2} n^2 \lambda_+^2 \int_{\X}\int_{\X} \prob{[\eta_n + \delta_z](B_{r_n(x,b_0)}(x)) \leq k-1, [\eta_n + \delta_x](B_{r_n(z,b_0)}(z)) \leq k-1} \\
 & \quad \qquad \qquad \qquad \times \ind\set{|x-z|> \min\{r_n(x,b_0),r_n(z,b_0)\}} \ind\set{|x-z| \leq r_n(x,b) + r_n(z,b)} \, \md x \, \md z \\
& =: E_{3,1}+E_{3,2}.
\end{align*}
We first consider $E_{3,1}$. Since $E_{3,1}=0$ for $k=1$, we can assume $k\geq 2$ in the sequel. We obtain that
\begin{align*}
E_{3,1} & \leq {2} n^2 \lambda_+^2 \int_{\X}\int_{\X} \prob{\eta_n (B_{r_n(x,b_0)}(x)) \leq k-2, \eta_n (B_{r_n(z,b_0)}(z)) \leq k-2} \\
 & \qquad \qquad \qquad \times \ind\set{|x-z|\leq \min\{r_n(x,b_0),r_n(z,b_0)\}} \ind\set{|x-z| \leq r_n(x,b) + r_n(z,b)} \, \md x \, \md z \\
& \leq {4} n^2 \lambda_+^2 \int_{\X}\int_{\X} \prob{\eta_n (B_{r_n(x,b_0)}(x)) \leq k-2, \eta_n (B_{r_n(z,b_0)}(z)) \leq k-2} \\
 & \qquad \qquad \qquad \times \ind\set{r_n(x,b_0) \leq r_n(z,b_0)} \ind\set{|x-z| \leq r_n(x,b) + r_n(z,b)} \, \md x \, \md z \\
& \leq {4} n^2 \lambda_+^2 \int_{\X}\int_{\X} \prob{\eta_n (B_{r_n(x,b_0)}(x)) \leq k-2} \prob{ \eta_n (B_{r_n(z,b_0)}(z)\setminus B_{r_n(x,b_0)}(x)) \leq k-2} \\
 & \qquad \qquad \qquad \times \ind\set{r_n(x,b_0) \leq r_n(z,b_0)} \ind\set{|x-z| \leq r_n(x,b) + r_n(z,b)} \, \md x \, \md z.
\end{align*}
For $x,z\in\X$ we have
$$
\prob{\eta_n (B_{r_n(x,b_0)}(x)) \leq k-2} =\sum_{i=0}^{k-2} \frac{(a_n+b_0)^i}{i!} e^{-(a_n+b_0)} \leq \frac{C_{b_0,k}}{n(a_n+b_0)}
$$
with some constant $C_{b_0,k}\in (0,\infty)$ only depending on $b_0$ and $k$ and
\begin{align*}
& \prob{ \eta_n (B_{r_n(z,b_0)}(z)\setminus B_{r_n(x,b_0)}(x)) \leq k-2} \\
& = \sum_{i=0}^{k-2} \frac{(n\bK(B_{r_n(z,b_0)}(z)\setminus B_{r_n(x,b_0)}(x)))^i}{i!} e^{-n\bK(B_{r_n(z,b_0)}(z)\setminus B_{r_n(x,b_0)}(x))} \\
& \leq C_k e^{-n\bK(B_{r_n(z,b_0)}(z)\setminus B_{r_n(x,b_0)}(x))/2}
\end{align*}
with some constant $C_k\in(0,\infty)$ only depending on $k$. 
By a short computation one can establish that there exists a dimension dependent constant $c_d\in(0,\infty)$ such that
$$
|B_1(0)\setminus B_1(v)| \geq c_d |v|
$$
for all $v\in B(0,2)$, see e.g. \cite[Equation (7.5)]{PenroseGoldstein2010}. From this and $r_n(x,b_0) \leq r_n(z,b_0)$ we deduce
\begin{align*}
\bK(B_{r_n(z,b_0)}(z)\setminus B_{r_n(x,b_0)}(x)) & \geq \lambda_{-} |B_{r_n(z,b_0)}(z)\setminus B_{r_n(x,b_0)}(x)| \\
& \geq \lambda_{-}  |B_{r_n(z,b_0)}(z)\setminus B_{r_n(z,b_0)}(x)| \geq \lambda_{-}  c_d r_n(z,b_0)^{d-1} |x-z|.
\end{align*}
Combining the previous estimates and using spherical coordinates we obtain
\begin{align*}
E_{3,1} & \leq {4} n^2 \lambda_+^2 \int_{\X}\int_{\X} \frac{C_{b_0,k} C_k}{n(a_n+b_0)} e^{-n \lambda_{-} c_d r_n(z,b_0)^{d-1} |x-z|/2}  \ind\set{|x-z| \leq r_n(x,b) + r_n(z,b)} \, \md x \, \md z \\
& \leq {4} \lambda_{+}^2C_{b_0,k}C_k\frac{n}{a_n+b_0} d\omega_d \int_{\X} \int_0^\infty e^{-n \lambda_{-} c_d r_n(z,b_0)^{d-1} s/2} s^{d-1} \, \md s \, \md z \\
& \leq {4} \lambda_{+}^2C_{b_0,k} C_k d\omega_d \frac{n}{a_n+b_0} \int_{\X} \frac{2^d}{(n \lambda_{-}c_d r_n(z,b_0)^{d-1})^d} \int_0^\infty e^{-u} u^{d-1} \, \md u  \, \md z.
\end{align*}
Since $n \lambda_{+} \omega_d r_n(z,b_0)^d\geq n\bK(B_{r_n(z,b_0)}(z)) = a_n + b_0$, we conclude that there exists a constant $C_{3,1}\in (0,\infty)$ only depending on $b_0$, $k$, $d$, $\lambda_+$ and $\lambda_-$ such that
\begin{equation}\label{e:kNNE31}
E_{3,1} \leq \frac{C_{3,1}}{(a_n+b_0)^d}.
\end{equation}
For $E_{3,2}$ we obtain
\begin{align*}
E_{3,2} & \leq {4} n^2 \lambda_+^2 \int_{\X}\int_{\X} \prob{\eta_n (B_{r_n(x,b_0)}(x)) \leq k-1, \eta_n(B_{r_n(z,b_0)}(z)) \leq k-1} \\
 & \quad \qquad \qquad \qquad \times \ind\set{|x-z|> r_n(x,b_0), r_n(z,b_0)\geq r_n(x,b_0)} \ind\set{|x-z| \leq r_n(x,b) + r_n(z,b)} \, \md x \, \md z \\
 & \leq {4} n^2 \lambda_+^2 \int_{\X}\int_{\X} \prob{\eta_n (B_{r_n(x,b_0)}(x)) \leq k-1}  \prob{\eta_n(B_{r_n(z,b_0)}(z)\setminus B_{r_n(x,b_0)}(x)) \leq k-1} \\
 & \quad \qquad \qquad \qquad \times \ind\set{|x-z|> r_n(x,b_0), r_n(z,b_0)\geq r_n(x,b_0)} \ind\set{|x-z| \leq r_n(x,b) + r_n(z,b)} \, \md x \, \md z.
\end{align*}
Let $x,z\in\X$. There exists a constant $C'_{b_0,k}\in(0,\infty)$ only depending on $b_0$ and $k$ such that
$$
\prob{\eta_n (B_{r_n(x,b_0)}(x)) \leq k-1} \leq \sum_{i=0}^{k-1} \frac{(a_n+b_0)^i}{i!} e^{-(a_n+b_0)} \leq \frac{C'_{b_0,k}}{n}.
$$
From $|x-z|> r_n(x,b_0)$ and the assumptions on $\bK$ it follows that
$$
\bK(B_{r_n(z,b_0)}(z)\setminus B_{r_n(x,b_0)}(x)) \geq \lambda_{-} |B_{r_n(z,b_0)}(z)\setminus B_{r_n(x,b_0)}(x)| \geq \frac{\lambda_{-}}{2} |B_{r_n(z,b_0)}(z)| \geq \frac{\lambda_{-}}{2\lambda_{+}} \bK(B_{r_n(z,b_0)}(z))
$$
so that
$$
n \bK(B_{r_n(z,b_0)}(z)\setminus B_{r_n(x,b_0)}(x)) \geq \gamma (a_n+b_0)
$$
with $\gamma:=\lambda_-/(2\lambda_+)$. Consequently, we can choose a constant $C''_{b_0,k}\in(0,\infty)$ only depending on $b_0$ and $k$ such that
$$
\prob{\eta_n(B_{r_n(z,b_0)}(z)\setminus B_{r_n(x,b_0)}(x)) \leq k-1} \leq C''_{b_0,k} \frac{(a_n+b_0)^{k-1}}{n^{\gamma}}.
$$
Moreover, we have $n \lambda_{-} \omega_d r_n(x,b)^d\leq n\bK(B_{r_n(x,b)}(x)) =a_n +b$ and $n \lambda_{-} \omega_d r_n(z,b)^d\leq n\bK(B_{r_n(z,b)}(z)) =a_n +b$. Combining all these estimates, we see that there exists a constant $C_{3,2}\in(0,\infty)$ only depending on $b_0$, $k$, $\lambda_+$ and $\lambda_-$ such that
\begin{equation}\label{e:kNNE32}
E_{3,2} \leq C_{3,2} n^2 \frac{1}{n} \frac{(a_n+b_0)^{k-1}}{n^{\gamma}} \frac{a_n+b}{n} = C_{3,2} \frac{(a_n+b_0)^{k-1}(a_n+b)}{n^{\gamma}}.
\end{equation}
Combining \eqref{e:knngt1}, \eqref{e:knngt2}, \eqref{e:knngt3}, \eqref{e:kNNE31} and \eqref{e:kNNE32} with Theorem \ref{t:qldep} yields that
\begin{align*}
& \dkr(\xi[\eta_n] \cap (\X \times (b_0,\infty)),\zeta \cap (\X \times (b_0,\infty))) \\
& \leq C_0\frac{\log \log n}{\log n} + C_1 e^{-b}\param{1 + \frac{\log \log n +b}{\log n}}^{k-1} + C_2 \frac{a_n + b}{n} + C_{3,1}\frac{1}{(a_n+b_0)^d} + C_{3,2}\frac{(a_n+b_0)^{k-1}(a_n+b)}{n^{\gamma}}.
\end{align*}
Now the choice $b=\log n$ proves \eqref{e:knnbd}.

Since convergence in KR distance implies convergence in distribution, we have $\xi[\eta_n] \cap (\X \times (b_0,\infty))\stackrel{d}{\to}\zeta \cap (\X \times (b_0,\infty))$ as $n\to\infty$. By the characterization of convergence in distribution of point processes in \cite[Theorem 16.16]{KallenbergFoundations}, this implies $\xi[\eta_n] \stackrel{d}{\to}\zeta $ as $n\to\infty$.
\end{proof}

In the next theorem, we consider the large $k$-nearest neighbor balls for an underlying binomial point process {under some assumptions on the density of the distribution of the points.} {For an underlying homogeneous binomial point process Poisson process convergence was established in \cite{Penrose1997}.} As mentioned before,  we are not aware of any {quantitative} result for binomial point processes even with a weaker rate of convergence or under weaker approximation distance in the literature.  {For a recent Poisson approximation result of nearest neighbor balls (i.e., $k=1$) of a binomial point process in $\R^d$ see \cite{GyoerfiHenzeWalk2019}.}
\begin{thm}
\label{t:knnapproxbin}
Let $d \geq 2$ and $k \geq 1$. Let $\beta_n$ be a binomial point process of $n$ points distributed according to a probability measure $\bK$ on $\X$ and assume that $\bK$ has a density $\lambda : \X \to (0,\infty)$ such that $0 < \lambda_- \le \lambda(x) \le \lambda_+ < \infty$ for all $x\in \X$. Let $\zeta$ be a Poisson process with intensity measure $\bM(\md x,\md t) = \lambda(x) \, \md x \, e^{-t} \, \md t$ on $\X \times \R$.  Then, for any $b_0 \in \R$, there exists a constant $C\in(0,\infty)$ only depending on $b_0$, $k$, $d$, $\lambda_+$ and $\lambda_-$ such that
\begin{equation}\label{e:knnbinbd}
\dkr(\xi[\beta_n] \cap (\X \times (b_0,\infty)),\zeta \cap (\X \times (b_0,\infty))) \leq C \frac{\log \log n}{\log n}
\end{equation}
for all $n\ge 3$. Moreover, $\xi[\beta_n] \stackrel{d}{\to} \zeta$ as $n \to \infty$.
\end{thm}

\begin{proof}
It suffices to prove \eqref{e:knnbinbd} and the remaining statement follows as in Theorem \ref{t:knnapprox}.  The rest of the proof will be concerned about proving \eqref{e:knnbinbd}.

We define $f$, $g$, $\mathcal{S}$, $\xi_0$, $r_n$ and $\varrho_n$ exactly as in the proof of Theorem \ref{t:knnapprox}. We also define $S_x=B_{r_n(x,b)(x)}$ for a fixed $b\in(\max\{0,b_0\},a_n {]}$, which will be chosen later. Our goal is to apply Theorem \ref{thm:binomial}. 

Throughout this proof we use several times the observation that
\begin{equation}\label{e:ApproximationExponential}
0 \leq e^{-y} - (1-y/n)^n \leq e^{-y} \frac{y^2}{n}
\end{equation}
for $y\in\R$ and $n\in\N$ such that $y/n\leq 1/2$. 
From now on we assume that $n$ is sufficiently large so that
\begin{equation}\label{e:Assumption_a_n}
0 \leq \frac{a_n+b_0}{n} \leq \frac{1}{4} \quad \text{and} \quad 0\leq \frac{2 a_n}{n}\leq \frac{1}{4}. 
\end{equation}
We denote by $\bL$ the intensity measure of $\xi[\beta_n]$. For $B\in\mathcal{X}$ and $u\in(b_0,\min\{\varrho_n,a_n\})$ we have
\begin{align*}
& \bL(B\times(u,\infty)) = n \int_B  \prob{ \beta_{n-1}(B_{r_n(x,u)}(x))\leq k-1 } \lambda(x) \md x \\
& = n \bK(B) \sum_{i=0}^{k-1} \binom{n-1}{i} \frac{(a_n+u)^i}{n^i} \bigg( 1 - \frac{a_n+u}{n} \bigg)^{n-1-i} = n \bK(B) \sum_{i=0}^{k-1} \frac{(n-1)_i}{n^i}  \frac{(a_n+u)^i}{i!} \bigg( 1 - \frac{a_n+u}{n} \bigg)^{n-1-i} .
\end{align*}
Thus $\bL$ has the density
\begin{align*}
\ell(x,u) & = -n \lambda(x)\sum_{i=1}^{k-1} \frac{(n-1)_i}{n^i}  \frac{(a_n+u)^{i-1}}{(i-1)!} \bigg( 1 - \frac{a_n+u}{n} \bigg)^{n-1-i} \\ & \quad + n \lambda(x) \sum_{i=0}^{k-1} \frac{(n-1)_i}{n^i} \frac{n-1-i}{n} \frac{(a_n+u)^i}{i!} \bigg( 1 - \frac{a_n+u}{n} \bigg)^{n-2-i} \\
& =: \lambda(x) q(u) 
\end{align*}
for $x\in\mathbb{X}$ and $u\in(b_0,\min\{\varrho_n,a_n\})$, whence 
\begin{align*}
| q(u) - e^{-u}| & \leq n \sum_{i=1}^{k-1}  \frac{(a_n+u)^{i-1}}{(i-1)!} \bigg( 1 - \frac{a_n+u}{n} \bigg)^{n-1-i} 
 + n \sum_{i=0}^{k-2} \frac{(a_n+u)^i}{i!} \bigg( 1 - \frac{a_n+u}{n} \bigg)^{n-2-i} \\
& \quad + \bigg|n \frac{(n-1)_{k-1}}{n^{k-1}} \frac{n-k}{n} \frac{(a_n+u)^{k-1}}{(k-1)!} \bigg( 1 - \frac{a_n+u}{n} \bigg)^{n-k-1} - e^{-u} \bigg| \\
& =: R_1+R_2+R_3.
\end{align*}
From \eqref{e:ApproximationExponential} and \eqref{e:Assumption_a_n} we deduce that
$$
n\bigg( 1 - \frac{a_n+u}{n} \bigg)^{n-k-1} \leq \frac{4^{k+1}}{3^{k+1}} n e^{-(a_n+u)}= \frac{4^{k+1} (k-1)!}{3^{k+1}} \frac{e^{-u}}{(\log n)^{k-1}}.
$$
Together with $|a_n+u|^j\leq 2^{j-1} (|a_n|^j + |u|^j)$ and $\sup_{v\in (b_0,\infty)} |v|^j e^{-v/2}<\infty$ for $j\in\mathbb{N}$ this yields
$$
R_1 \leq \frac{c_1 e^{-u/2}}{\log n} \quad \text{and} \quad R_2 \leq \frac{c_2 e^{-u/2}}{\log n}, 
$$
where $c_1,c_2\in(0,\infty)$ depend only on $b_0$ and $k$. For $R_3$ we have
\begin{align*}
R_3 & \leq \bigg| \frac{(n-1)_{k-1}}{n^{k-1}} \frac{n-k}{n} - 1 \bigg| \frac{n (a_n+u)^{k-1}}{(k-1)!} \bigg( 1 - \frac{a_n+u}{n} \bigg)^{n-k-1} \\
& \quad + \bigg| \frac{n (a_n+u)^{k-1}}{(k-1)!} - \frac{n a_n^{k-1}}{(k-1)!} \bigg|  \bigg( 1 - \frac{a_n+u}{n} \bigg)^{n-k-1} \\
& \quad + \bigg| 1 - \bigg( 1 - \frac{a_n+u}{n} \bigg)^{k+1} \bigg| \frac{n a_n^{k-1}}{(k-1)!} \bigg( 1 - \frac{a_n+u}{n} \bigg)^{n-k-1} + \bigg| \frac{n a_n^{k-1}}{(k-1)!} \bigg( 1 - \frac{a_n+u}{n} \bigg)^{n} - e^{-u} \bigg| .
\end{align*}
By the same arguments as used for $R_1$ and $R_2$ and
$$
\bigg| \frac{n a_n^{k-1}}{(k-1)!} \bigg( 1 - \frac{a_n+u}{n} \bigg)^{n} - e^{-u} \bigg| \leq \frac{n a_n^{k-1}}{(k-1)!} e^{-a_n-u} \frac{(a_n+u)^2}{n} + \bigg| \frac{n a_n^{k-1}}{(k-1)!} e^{-a_n-u} - e^{-u} \bigg|,
$$
which follows from \eqref{e:ApproximationExponential}, we obtain
$$
R_3 \leq \frac{c_3 e^{-u/2}}{n} + \frac{c_4 e^{-u/2}}{\log n} + \frac{c_5 e^{-u/2} a_n}{n} + \frac{c_6 e^{-u} a_n^2}{n} + \frac{c_7 e^{-u} \log\log n}{\log n}
$$
with constants $c_3,c_4,c_5,c_6,c_7\in(0,\infty)$ depending only on $b_0$ and $k$. Let $R(u)$ denote the sum of the upper bounds for $R_1$, $R_2$ and $R_3$ and let $\bM_0$ be the intensity measure of $\zeta\cap (\X \times (b_0,\infty))$. The estimates above imply that
\begin{equation}
\label{e:dtvlmb}
d_{TV}(\bL, \bM_0) \leq \int_{b_0}^\infty R(u) \, du + \bL( { \X\times } (\min\{\varrho_n,a_n\},\infty)) + \bM_0( { \X\times } (\min\{\varrho_n,a_n\},\infty)) \leq C_0' \frac{\log\log n}{\log n}
\end{equation}
with a constant $C_0'\in(0,\infty)$ depending on $b_0$, $k$, $d$ and $\lambda_-$.

In the following we compute the terms $E_1,\hdots,E_6$ from Theorem \ref{thm:binomial}. Using the above expression for $\bL$, we obtain that
\begin{align}
\label{e:E1bin} 
E_1 & \leq  2n \int_B  \prob{ \beta_{n-1}(B_{r_n(x,b)}(x))\leq k-1 } \lambda(x) \md x  \leq C_1' e^{-b}\param{1 + \frac{(k-1)\log \log n +b}{\log n}}^{k-1}
\end{align}
with $C_1':=2k! \big(\frac{4}{3}\big)^{k+1}$.
For $n\in\N$ and $x\in\mathbb{X}$ define $U_{x}=B_{r_n(x,b_0)}(x)$. It follows from \eqref{e:ApproximationExponential} {and \eqref{e:Assumption_a_n}} that
\begin{equation}\label{e:gbetan-1}
\begin{split}
\mean{g(x,\beta_{n-1}+\delta_x)} & \leq \prob{\beta_{n-1}(U_x)\leq k-1} 
= \sum_{i=0}^{k-1} \binom{n-1}{i} \bigg(\frac{a_n+b_0}{n} \bigg)^i \bigg(1-\frac{a_n+b_0}{n} \bigg)^{n-1-i} \\
& \leq \frac{4^k}{3^k} \sum_{i=0}^{k-1}  \frac{(a_n+b_0)^i}{i!} e^{-(a_n+b_0)} \leq \frac{\tilde{c}_1}{n}
\end{split}
\end{equation}
with a constant $\tilde{c}_1\in(0,\infty)$ depending only on $b_0$ and $k$. For $x,y,y_1,y_2\in\mathbb{X}$ {with $S_x\cap U_{y_1}=\emptyset$} we obtain analogously that there exists a constant $\tilde{c}_2\in(0,\infty)$ depending only on $b_0$ and $k$ such that
\begin{align}\label{e:gbetan-2}
\mean{g(x,\beta_{n-2}+\delta_x+\delta_y)} & \leq \mean{g(x,\beta_{n-2}+\delta_x)} \leq \prob{\beta_{n-2}(U_x)\leq k-1} \leq  \frac{\tilde{c}_2}{n},\\
 \label{e:gtildebetan-2}
\mean{g(y_1,\widetilde{\beta}_{x,n-2}+\delta_{y_1}+\delta_{y_2})} & \leq \prob{\widetilde{\beta}_{x,n-2}(U_{y_1})\leq k-1} \leq \prob{\beta_{n-2}(U_{y_1})\leq k-1} \leq \frac{\tilde{c}_2}{n}
\end{align}
and
\begin{equation} \label{e:gtildebetan-1}
{\mean{g(y_1,\widetilde{\beta}_{x,n-1}+\delta_{y_1})} \leq \prob{\widetilde{\beta}_{x,n-1}(U_{y_1})\leq k-1} \leq \prob{\widetilde{\beta}_{x,n-2}(U_{y_1})\leq k-1} \leq \frac{\tilde{c}_2}{n}.}
\end{equation}
For $x,y\in\mathbb{X}$ and $i,j\in\{0,1,\hdots,k-1\}$ we have
\begin{align*}
& \prob{\beta_{n-2}(U_{x})=i, \beta_{n-2}(U_{y})=j} \\
& = \sum_{s=0}^{\min\{i,j\}} \mathbb{P}\big(\beta_{n-2}(U_{x}\cap U_{y}^c)=i-s, \beta_{n-2}(U_{x}^c\cap U_{y})=j-s, \beta_{n-2}(U_{x}\cap U_{y})=s \big) \\
& = \sum_{s=0}^{\min\{i,j\}} \frac{(n-2)_{i+j-s}}{(i-s)! (j-s)! s!} \bK(U_{x}\cap U_{y}^c)^{i-s} \bK(U_{x}^c\cap U_{y})^{j-s} \bK(U_{x}\cap U_{y})^{s} \big( 1- \bK(U_{x}\cup U_{y}) \big)^{n-2-i-j+s} \\
& \leq 4^k \sum_{s=0}^{\min\{i,j\}} \frac{(n\bK(U_{x}\cap U_{y}^c))^{i-s}}{(i-s)!} \frac{(n\bK(U_{x}^c\cap U_{y}))^{j-s}}{(j-s)!} \frac{(n\bK(U_{x}\cap U_{y}))^{s}}{s!} \exp\big(-n\bK(U_{x}\cup U_{y}) \big) \\
& = 4^k \sum_{s=0}^{\min\{i,j\}} \mathbb{P}\big(\eta_{n}(U_{x}\cap U_{y}^c)=i-s, \eta_{n}(U_{x}^c\cap U_{y})=j-s, \eta_{n}(U_{x}\cap U_{y})=s \big) \\
& = 4^k \prob{\eta_{n}(U_{x})=i, \eta_{n}(U_{y})=j},
\end{align*}
where the inequality follows from \eqref{e:ApproximationExponential} and \eqref{e:Assumption_a_n}.
This implies that
\begin{equation}\label{e:Productgs}
\begin{split}
\mean{g(x,\beta_{n-2}+\delta_x+\delta_y) g(y,\beta_{n-2}+\delta_x+\delta_y)}
& \leq \prob{[\beta_{n-2}+\delta_y](U_{x})\leq k-1, [\beta_{n-2}+\delta_x](U_{y})\leq k-1} \\
& \leq 4^k \prob{[\eta_{n}+\delta_y](U_{x})\leq k-1, [\eta_{n}+\delta_x](U_{y})\leq k-1}
\end{split}
\end{equation}
and if $S_{x}\cap S_{y}=\emptyset$,
\begin{equation}\label{e:gSquareDisjoint}
\begin{split}
& \mean{g(x,\beta_{n-2}+\delta_x) g(y,\beta_{n-2}+\delta_y)} 
 = \mean{g(x,\beta_{n-2}+\delta_x+\delta_y) g(y,\beta_{n-2}+\delta_x+\delta_y)} \\
& \leq 4^k \prob{(\eta_{n}+\delta_y)(U_{x})\leq k-1, (\eta_{n}+\delta_x)(U_{y})\leq k-1}
 \leq 4^k \prob{\eta_{n}(U_{x})\leq k-1, \eta_{n}(U_{y})\leq k-1} \\
&
= 4^k \prob{\eta_{n}(U_{x})\leq k-1} \prob{\eta_{n}(U_{y})\leq k-1} \leq \frac{\tilde{c}_3}{n^2}
\end{split}
\end{equation}
with a constant $\tilde{c}_3\in(0,\infty)$ depending only on $b_0$ and $k$. {A similar computation shows that there exists a constant $\tilde{c}_4\in(0,\infty)$ depending only on $b_0$ and $k$ such that 
\begin{equation}\label{eqn:gSquareDisjointn-3}
\mean{g(x,\beta_{n-3}+\delta_x+\delta_z) g(y,\beta_{n-3}+\delta_y)} \leq \mean{g(x,\beta_{n-3}+\delta_x) g(y,\beta_{n-3}+\delta_y)} \leq \frac{\tilde{c}_4}{n^2}
\end{equation}
for $x,y,z\in \X$ with $S_x\cap S_y = \emptyset$. In the sequel, we will use that $\tilde{g}\leq g$.}
It follows from \eqref{e:gbetan-1} and
\begin{align*}
& \ind\{S_{x}\cap S_{y}\neq \emptyset\}\leq \ind\{|x-y|\leq r_n(x,b)+r_n(y,b)\} \leq \ind\{|x-y|\leq 2r_n(x,b)\} + \ind\{|x-y|\leq 2r_n(y,b)\}
\end{align*}
that
\begin{align}
\label{e:E2b} E_2 & \leq 2 \tilde{c}_1^2 \int_{\mathbb{X}^2} \ind\{S_{x}\cap S_{y}\neq \emptyset\} \bK^2(\md (x,y)) \leq 4 \tilde{c}_1^2 \int_{\mathbb{X}} \bK(B_{2r_n(x,b)}(x)) \bK(\md x) 
  \leq 4 \tilde{c}_1^2 2^d \frac{\lambda_+}{\lambda_-} \frac{a_n+b}{n}.
\end{align}
From \eqref{e:Productgs} and the same arguments as in the Poisson case (see \eqref{e:kNNE31} and \eqref{e:kNNE32}) we obtain that
\begin{align}
E_3 & \leq 2\cdot 4^k n^2 \int_{\mathbb{X}^2} \ind\{S_{x}\cap S_{y} \neq \emptyset\} \prob{[\eta_{n}+\delta_y](U_{x})\leq k-1, [\eta_{n}+\delta_x](U_{y})\leq k-1} \bK^2(\md(x,y)) \no \\
\label{e:E3b} & \leq C_{3,1}'\frac{1}{(a_n+b_0)^d} + C_{3,2}' \frac{(a_n+b_0)^{k-1}(a_n+b)}{n^{\gamma}}
\end{align}
with constants $C_{3,1}',C_{3,2}',\gamma\in(0,\infty)$ only depending on $b_0$, $k$, $d$, $\lambda_+$ and $\lambda_-$.
Because of \eqref{e:gbetan-1} and \eqref{e:gbetan-2} we have, for $x\in\mathbb{X}$,
\begin{align*}
& (1+n\bK(S_{x})) \mean{g(x,\beta_{n-1}+\delta_x)} + n \int_{S_{x}} \mean{g(x,\beta_{n-2}+\delta_x+\delta_z)} \bK(\md z) \leq \frac{\tilde{c}_1 (1+a_n+ b)}{n} + \frac{\tilde{c}_2 (a_n + b)}{n}.
\end{align*}
From $b\leq a_n$ and \eqref{e:Assumption_a_n} we derive for all $x\in\X$ that $\bK(S_x)\leq \frac{2a_n}{n}\leq\frac{1}{4}$, whence for $\mathbf{Q}=\bK$ the measure $\mathbf{Q}_x$ in Theorem \ref{thm:binomial} is dominated by $\frac{4}{3}\bK$. Together with \eqref{e:gtildebetan-2} and \eqref{e:gtildebetan-1} this implies that
\begin{align}
E_4 & \leq 2n\bigg( \frac{\tilde{c}_1 (1+a_n+ b)}{n} + \frac{\tilde{c}_2 (a_n + b)}{n} \bigg) \bigg( \frac{4}{3} \frac{\tilde{c}_2}{n} + \frac{16}{9} n \frac{\tilde{c}_2}{n} \int_{\mathbb{X}} \bK(S_{y_1}) \bK(\md y_1) \bigg)\no  \\
\label{e:E4b} & \leq 2\bigg( \tilde{c}_1(1+a_n+ b) + \tilde{c}_2 (a_n + b) \bigg)  \bigg( \frac{4}{3} \frac{\tilde{c}_2}{n} + \frac{16}{9} \tilde{c}_2 \frac{a_n + b}{n} \bigg).
\end{align}
By \eqref{e:gbetan-2} we have
\begin{align}
\label{e:E5b}
E_5 & \leq \frac{8}{3} (a_n + b)^2 n \frac{\tilde{c}_2^2}{n^2} = 2 \tilde{c}_2^2 \frac{(a_n + b)^2}{n}.
\end{align}
It follows from \eqref{e:gSquareDisjoint} {and \eqref{eqn:gSquareDisjointn-3}} that
\begin{align}
\label{e:E6b}
E_6 & {\leq 2n^2 \frac{\tilde{c}_3}{n^2}  \frac{4(a_n + b)}{3n} + 2n^3 \frac{\tilde{c}_4}{n^2}  \frac{4(a_n + b)^2}{3n^2}  = \frac{8}{3} \big( \tilde{c}_3 + \tilde{c}_4 (a_n + b) \big) \frac{a_n + b}{n}.}
\end{align}
Choosing $b=a_n$ and combining the bounds \eqref{e:dtvlmb},  \eqref{e:E1bin}, \eqref{e:E2b},  \eqref{e:E3b},  \eqref{e:E4b},  \eqref{e:E5b} and \eqref{e:E6b} with Theorem \ref{thm:binomial} proves \eqref{e:knnbinbd}.
\end{proof}

\section*{Acknowledgements}
OB was supported in part by the Israel Science Foundation, Grant 1965/19.  MS was supported by the Swiss National Science Foundation Grant No.\ 200021\_175584. DY's research was funded by DST-INSPIRE Faculty Award,  SERB-MATRICS Grant and CPDA from the Indian Statistical Institute.  DY wishes to thank Moritz Otto for explaining his results in advance.   



\end{document}